\documentclass{dcds}
\usepackage{amsmath}
  \usepackage{paralist}
  \usepackage{graphics} 
  \usepackage{epsfig} 
\usepackage{graphicx}  \usepackage{epstopdf}
 \usepackage[colorlinks=true]{hyperref}
\hypersetup{urlcolor=blue, citecolor=red}

\def\currentvolume{38}
 \def\currentissue{10}
  \def\currentyear{2018}
   \def\currentmonth{October}
    \def\ppages{X--XX}
        \def\DOI{10.3934/dcds.2018216}

\newtheorem{theorem}{Theorem}[section]

\newtheorem{lemma}[theorem]{Lemma}
\newtheorem{proposition}{Proposition}

\theoremstyle{definition}

\newtheorem{remark}{Remark}

\title[Convergence of the perturbed compositional gradient flow] 
      {A convergence analysis of the perturbed compositional gradient flow: Averaging principle and normal deviations}

\author[Wenqing Hu and Chris Junchi Li]{}


 \email{huwen@mst.edu}
 \email{junchi.li.duke@gmail.com}

\thanks{The first author is supported by an University of Missouri Research Board (UMRB) grant.}

\thanks{$^*$ Corresponding author: Wenqing Hu}

\renewcommand{\abstractname}{Abstract}
\renewcommand{\figurename}{Fig.}

\begin{document}
\maketitle

\centerline{\scshape Wenqing Hu$^*$}
\medskip
{\footnotesize
 \centerline{Department of Mathematics and Statistics, Missouri University of Science and Technology}
   \centerline{(formerly University of Missouri, Rolla)}
   \centerline{Rolla, MO 65409--0020, USA}
} 

\medskip

\centerline{\scshape Chris Junchi Li}
\medskip
{\footnotesize
 \centerline{Department of Operations Research and Financial Engineering, Princeton University}
   \centerline{Princeton, NJ 08544, USA}
}

\bigskip

 \centerline{(Communicated by Sandra Cerrai)}

\begin{abstract}
We consider in this work a system of two stochastic differential equations named the perturbed compositional gradient flow. By introducing a separation of fast and slow scales of the two equations, we show that the limit of
the slow motion is given by an averaged ordinary differential equation. We then
demonstrate that the deviation of the slow motion from the averaged equation, after proper rescaling, converges to a stochastic process
with Gaussian inputs.
This indicates that the slow motion can be approximated in the weak sense by a standard perturbed gradient flow or the continuous-time stochastic gradient descent algorithm that solves the optimization problem for a composition of two functions.
As an application, the perturbed compositional gradient flow corresponds to the diffusion limit of the
Stochastic Composite Gradient Descent
(SCGD) algorithm for minimizing a composition of two expected-value functions in the optimization literatures. For the strongly convex case, such an analysis implies that the SCGD algorithm has the same convergence time asymptotic as the classical stochastic gradient descent algorithm.
Thus it validates, at the level of continuous approximation, the effectiveness of using the SCGD algorithm in the strongly convex case.
\end{abstract}

\section{Introduction}

In this work we target at analyzing a system of two stochastic differential equations called the \textit{perturbed compositional gradient flow}, which takes the form
\begin{equation}\label{Eq:DiffusionLimit}
\left\{\begin{array}{ll}
dy(t)=-\varepsilon y(t) dt +\varepsilon \mathbf{E} g_w(x(t))dt+ \varepsilon \Sigma_1(x(t))dW_t^1  \ , & y(0)=y_0 \ ,
\\
dx(t)=-\eta \mathbf{E}\widetilde{\nabla} g_w(x(t))\nabla f_{v}(y(t))dt+\eta \Sigma_2(x(t), y(t))dW_t^2 \ , & x(0)=x_0 \ .
\end{array}\right.
\end{equation}

Here $(w,v)$ follows a certain distribution on an index set $\mathcal{D}$;
$f_v: \mathbb{R}^m\rightarrow \mathbb{R}$ and $g_w: \mathbb{R}^n\rightarrow \mathbb{R}^m$ are assumed to be in $\mathbf{C}^{(4)}$; the vector $\nabla f_{v}(y)$
is the gradient column $m$--vector of $f_{v}$ evaluated at $y$
and the matrix $\widetilde{\nabla} g_{w}(x)$ is the $n\times m$ matrix formed by the gradient column
$n$--vector of each of the $m$ components of $g_{w}$ evaluated at $x$; $\varepsilon>0$ and $\eta>0$ are two small parameters.
We assume that the functions $f$ and $g$
are supported on some compact subsets of $\mathbb{R}^m$ and $\mathbb{R}^n$, respectively
\footnote{A further discussion of this assumption is provided in Remark 3 of Section 5.}.

The two Brownian motions $W_t^1$ and $W_t^2$ are independent standard Brownian motions moving in the spaces $\mathbb{R}^m$ and $\mathbb{R}^n$, respectively.
Here the diffusion matrix $\Sigma_1(x)$ satisfies
$$\Sigma_1(x)\Sigma_1^T(x)=\mathbf{E}\left[
\left(g_{w}(x)-\mathbf{E} g_w(x)\right)\left(g_{w}(x)-\mathbf{E} g_w(x)\right)^T
\right]\ ,$$
and the diffusion matrix $\Sigma_2(x,y)$ satisfies
$$\begin{array}{ll}
\Sigma_2(x,y)\Sigma_2^T(x,y)=&\mathbf{E}\left[
\left(\widetilde{\nabla} g_{w}(x)\nabla f_{v}(y)
-\mathbf{E}\widetilde{\nabla} g_{w}(x)\nabla f_{v}(y)\right)
\right.
\\
&\qquad \qquad \qquad \cdot
\left.
 \left(\widetilde{\nabla} g_{w}(x)\nabla f_{v}(y)
-\mathbf{E}\widetilde{\nabla} g_{w}(x)\nabla f_{v}(y)\right)^T
\right]\ ,
\end{array}$$
and both matrices are assumed to be non--degenerate for any choice of $x\in \mathbb{R}^n$ and $y\in \mathbb{R}^m$.

\subsection{Coupled fast--slow dynamics and averaging principle.}\label{ssec:coupled}

It turns out that, by an appropriate choice of the step size parameters,
the perturbed compositional gradient flow exhibits a fast--slow dynamics. To see this, we perform a change into the fast time scale for \eqref{Eq:DiffusionLimit} and we let $t \mapsto t/\eta$.
Then we have, for the time--changed process
$(X^{\varepsilon,\eta}(t), Y^{\varepsilon,\eta}(t))=(x(t/\eta), y(t/\eta))$, that
\begin{equation}\label{Eq:DiffusionLimitTimeChanged}
\left\{\begin{array}{l}
dY^{\varepsilon,\eta}(t)=-\dfrac{\varepsilon}{\eta} Y^{\varepsilon,\eta}(t) dt +\dfrac{\varepsilon}{\eta} \mathbf{E} g_w(X^{\varepsilon,\eta}(t))dt+
\dfrac{\varepsilon}{\sqrt{\eta}} \Sigma_1(X^{\varepsilon,\eta}(t))dW_t^1  \ ,
\\
 \qquad \qquad \qquad \qquad \qquad \qquad \qquad \qquad \qquad Y^{\varepsilon,\eta}(0)=y_0 \ ,
\\
dX^{\varepsilon,\eta}(t)=- \mathbf{E}\widetilde{\nabla} g_w(X^{\varepsilon,\eta}(t))\nabla f_{v}(Y^{\varepsilon,\eta}(t))dt+
\sqrt{\eta}\Sigma_2(X^{\varepsilon,\eta}(t), Y^{\varepsilon,\eta}(t))dW_t^2 \ ,
\\
 \qquad \qquad \qquad \qquad \qquad \qquad \qquad \qquad \qquad X^{\varepsilon,\eta}(0)=x_0 \ .
\end{array}\right.
\end{equation}

We will set the vectors
\begin{equation}\label{Eq:Quantity:B1}
B_1(X)=\mathbf{E} g_w(X)\in \mathbb{R}^m
\end{equation}
and
\begin{equation}\label{Eq:Quantity:B2}
B_2(X,Y)=- \mathbf{E}\widetilde{\nabla} g_w(X)\nabla f_{v}(Y)\in \mathbb{R}^n \ ,
\end{equation}
and the matrices
\begin{equation}\label{Eq:Quantity:A1}
A_1(X)=\Sigma_1(X)\Sigma_1^T(X)\in \mathbb{R}^m\otimes \mathbb{R}^m \ ,
\end{equation}
and
\begin{equation}\label{Eq:Quantity:A2}
A_2(X,Y)=\Sigma_2(X,Y)\Sigma_2^T(X,Y)\in \mathbb{R}^n\otimes \mathbb{R}^n \ .
\end{equation}

From our assumption on $f$ and $g$ we know that the vectors $B_1(X)$, $B_2(X,Y)$
and the matrices $A_1(X)$ and $A_2(X,Y)$ contain bounded coefficients together with their first derivatives, so that these quantities
are also uniformly Lipschitz continuous with respect to their arguments.

One can write system \eqref{Eq:DiffusionLimitTimeChanged} as
{\small\begin{equation}\label{Eq:DiffusionLimitTimeChangedStandardForm}
\left\{\begin{array}{ll}
dY^{\varepsilon,\eta}(t)=-\dfrac{\varepsilon}{\eta} Y^{\varepsilon,\eta}(t) dt +\dfrac{\varepsilon}{\eta} B_1(X^{\varepsilon,\eta}(t))dt+
\dfrac{\varepsilon}{\sqrt{\eta}} \Sigma_1(X^{\varepsilon,\eta}(t))dW_t^1, & Y^{\varepsilon,\eta}(0)=y_0,
\\
dX^{\varepsilon,\eta}(t)= B_2(X^{\varepsilon,\eta}(t), Y^{\varepsilon,\eta}(t))dt+ \sqrt{\eta}\Sigma_2(X^{\varepsilon,\eta}(t), Y^{\varepsilon,\eta}(t))dW_t^2, &
X^{\varepsilon,\eta}(0)=x_0.
\end{array}\right.
\end{equation}}

Fix $\varepsilon>0$ and let $\eta \rightarrow 0$. Thus $\dfrac{\varepsilon}{\eta}\rightarrow \infty$ as $\eta\rightarrow 0$. Then system \eqref{Eq:DiffusionLimitTimeChangedStandardForm}
is a standard fast--slow system
of stochastic differential equations.
In fact, the $Y$--component of the system \eqref{Eq:DiffusionLimitTimeChanged}
can be written as
\begin{align*}
dY^{\varepsilon,\eta}(t)
&=
-\dfrac{\varepsilon}{\eta} Y^{\varepsilon,\eta}(t) dt +\dfrac{\varepsilon}{\eta} \mathbf{E} g_w(X^{\varepsilon,\eta}(t))dt+
\sqrt{\varepsilon}\left(\dfrac{\varepsilon}{\eta}\right)^{1/2} \Sigma_1(X^{\varepsilon,\eta}(t))dW_t^1
 \ ,
 \\
 \ Y^{\varepsilon,\eta}(0)
 &=
 y_0
  \ ,
\end{align*}
so that as $\dfrac{\varepsilon}{\eta}\rightarrow \infty$, the $Y$ motion is running at a fast speed the following Ornstein--Uhlenbeck process
(OU process for short, see \cite[Exercise 5.5]{[Oksendal]})
\begin{equation}\label{Eq:FastMotionDeterministicSystemRandomPerturbation}
d\mathfrak{y}^{X,\varepsilon}(t)=-\mathfrak{y}^{X,\varepsilon}(t)dt + \mathbf{E} g_w(X)dt+\sqrt{\varepsilon}\Sigma_1(X)dW_t^2 \ , \ \mathfrak{y}^{X,\varepsilon}(0)=y_0 \ .
\end{equation}

The invariant measure $\mu^{X,\varepsilon}(dY)$ of the (multidimensional) OU process $\mathfrak{y}^{X,\varepsilon}(t)$ is a Gaussian measure with mean $\mathbf{E} g_w(X)$
and covariance matrix $\dfrac{\varepsilon}{2} \Sigma_1(X)\Sigma_1^T(X)$:
\begin{equation}\label{Eq:InvariantMeasureOUSmallDiffusion}
\mu^{X,\varepsilon}(dY)\sim \mathcal{N}\left(\mathbf{E} g_w(X), \dfrac{\varepsilon}{2} \Sigma_1(X)\Sigma_1^T(X)\right) \ .
\end{equation}

Let us introduce the operator
\begin{equation}\label{Eq:BarEpsOperator}
\overline{q(X,Y)}^{\varepsilon}=\overline{q(X,Y)}^{\varepsilon}(X)=\int_{\mathbb{R}^m}q(X,Y)\mu^{X,\varepsilon}(dY) \ ,
\end{equation}
where $q(X,Y)$ can be scalar, vector or matrix--valued functions with arguments $X$ and $Y$.

As the \textit{fast motion} $Y^{\varepsilon,\eta}(t)$ process is running at a high speed, the process $X^{\varepsilon,\eta}(t)$ in \eqref{Eq:DiffusionLimitTimeChangedStandardForm}
plays the role of the \textit{slow motion}. That is to say, $X^{\varepsilon,\eta}(t)$ changes very little, and thus could be viewed as frozen,
 during a small time interval in which $Y^{\varepsilon,\eta}(t)$ is running very fast. Roughly speaking,
 in the dynamics of $X^{\varepsilon,\eta}(t)$, the fast component $Y^{\varepsilon,\eta}(t)$ can be
 replaced by the invariant measure of $\mathfrak{y}^{X^{\varepsilon,\eta}(t),\varepsilon}$ with frozen $X^{\varepsilon,\eta}(t)$.
 This heuristic supports the following asymptotic picture:
 as $\varepsilon>0$ fixed and $\eta\rightarrow 0$, thus $\dfrac{\varepsilon}{\eta}\rightarrow \infty$, one can approximate the slow process
 $X^{\varepsilon,\eta}(t)$ in \eqref{Eq:DiffusionLimitTimeChangedStandardForm} by an \textit{averaged} process $X^\varepsilon(t)$ satisfying
\begin{equation}\label{Eq:AveragedSlowMotionGaussianMeasureDeterministic}
dX^\varepsilon(t)=\overline{B_2(X^\varepsilon(t), Y)}^\varepsilon(X^\varepsilon(t))dt \ , \ X^\varepsilon(0)=x_0 \ .
\end{equation}

The approximation of $X^{\varepsilon,\eta}(t)$ by $X^\varepsilon(t)$ is the content of the
classical \textit{averaging principle} and was discussed in many literatures (see e.g. \cite{[Khasminskii1968AveragingSDE]}, \cite{[Khasminskii1966SmallParameterDE]}, \cite[Chapter 7]{[FW book]}).
In this paper we will show that (see Proposition \ref{Proposition:ConvergenceXEpsEtaToXEps}),
as $\varepsilon>0$ is fixed and set $\eta \rightarrow 0$, for $0\leq t\leq T$ we have
\begin{equation}\label{Eq:ApproximationSlowProcessByAveragedMotion}
\sup\limits_{0\leq t\leq T}\mathbf{E}|X^{\varepsilon,\eta}(t)-X^\varepsilon(t)|_{\mathbb{R}^n}^2\rightarrow 0 \ .
\end{equation}
This justifies the approximation of the averaged motion $X^\varepsilon(t)$ to the slow process $X^{\varepsilon,\eta}(t)$.

It turns out, as we will prove quantitatively in Lemma
\ref{Lemma:ErrorSmallEpsAveragedQuantity} below, that when $\varepsilon\rightarrow 0$,
\begin{equation}\label{Eq:ConvergenceSmallEpsAveragedQuantity}
\overline{q(X,Y)}^\varepsilon- q(X, \mathbf{E} g_w(X)) \approx \mathcal{O}(\sqrt{\varepsilon}) \ .
\end{equation}

Therefore as $\varepsilon\rightarrow 0$, by \eqref{Eq:Quantity:B2} we see that
$$\overline{B_2(X, Y)}^\varepsilon(X)\approx -\mathbf{E} \widetilde{\nabla} g_w(X)\nabla f_v(\mathbf{E} g_w(X))+\mathcal{O}(\sqrt{\varepsilon}) \ .$$
Thus as $\varepsilon \rightarrow 0$, the process $X^\varepsilon(t)$ approximates another process $\bar{X}(t)$ that solves an ordinary differential equation:
\begin{equation}\label{Eq:AveragedSlowMotionGD}
d\bar{X}(t)=-\mathbf{E} \widetilde{\nabla} g_w(\bar{X}(t))\nabla f_v(\mathbf{E} g_w(\bar{X}(t)))dt \ , \ \bar{X}(0)=x_0 \ ,
\end{equation}
with an error of $\mathcal{O}(\sqrt{\varepsilon})$. In fact, equation \eqref{Eq:AveragedSlowMotionGD} can be viewed as a gradient flow, which is the perturbed gradient flow with no stochastic noise terms \cite{[hu2017fast]}. Such averaging principle can hence explain why we call \eqref{Eq:DiffusionLimit} the perturbed compositional gradient flow.

\subsection{A sharper rate via normal deviation.}\label{ssec:sharper}

One major drawback of the classical averaging principle is that, the approximation $X^{\varepsilon,\eta}(t)\rightarrow X^\varepsilon(t)$ as $\eta\rightarrow 0$
 in \eqref{Eq:ApproximationSlowProcessByAveragedMotion} can only identify the deterministic drift, and thus the small diffusion part
 in the equation for $X^{\varepsilon,\eta}(t)$ vanishes as $\eta\rightarrow 0$.
To overcome this difficulty, let us consider the deviation $X^{\varepsilon,\eta}(t)-X^{\varepsilon}(t)$ and we rescale it by a factor of $\sqrt{\eta}$.
Thus we consider the process
\begin{equation}\label{Eq:RescaledDeviation}
Z^{\varepsilon,\eta}(t)=\dfrac{X^{\varepsilon,\eta}(t)-X^{\varepsilon}(t)}{\sqrt{\eta}} \ .
\end{equation}

We will show that (see Proposition \ref{Proposition:NormalApproximationWeakConvergence}), as $\eta \rightarrow 0$, the process $Z^{\varepsilon,\eta}(t)$ converges weakly to random process $Z_t^\varepsilon$. The process $Z_t^\varepsilon$ has its deterministic drift part and is driven by two mean $0$ Gaussian processes
carrying explicitly calculated covariance structures.
This implies that, roughly speaking, from \eqref{Eq:RescaledDeviation} we can expand
\begin{equation}\label{Eq:ExpansionNormalApproximation}
X^{\varepsilon,\eta}(t)\stackrel{\mathcal{D}}{\approx} X^\varepsilon(t)+\sqrt{\eta}Z_t^\varepsilon \ ,
\end{equation}
as $\eta\rightarrow 0$. Here $\stackrel{\mathcal{D}}{\approx}$ means approximate equality of probability distributions.
In fact, such approximate expansions have been introduced in the classical program under the context of stochastic climate models (see \cite{[Hasselmann1976]}, \cite[equation (4.8)]{[LArnold]}), and in physics this is also known as the \textit{Van Kampen's approximation} (see \cite{[VanKampen]}).

Therefore by \eqref{Eq:ConvergenceSmallEpsAveragedQuantity}, \eqref{Eq:AveragedSlowMotionGD} and \eqref{Eq:ExpansionNormalApproximation}
we know that the slow motion $X^{\varepsilon,\eta}(t)$ in \eqref{Eq:DiffusionLimitTimeChangedStandardForm} (or \eqref{Eq:DiffusionLimitTimeChanged})
has an expansion around the GD algorithm in \eqref{Eq:AveragedSlowMotionGD}:
\begin{equation}\label{Eq:ExpansionNormalApproximationGD}
X^{\varepsilon,\eta}(t)\stackrel{\mathcal{D}}{\approx} \bar{X}(t)+\mathcal{O}(\sqrt{\varepsilon})+\sqrt{\eta}Z_t^\varepsilon \ .
\end{equation}

Let us introduce the process $\mathrm{X}^{\varepsilon,\eta}(t)$ as
the following fast time-scale version of the perturbed gradient flow \cite{[hu2017fast]}, \cite{[JunchiEtAlDiffusionApproximation]}:
\begin{equation}\label{Eq:AveragedSlowMotionSGD}
d\mathrm{X}^{\varepsilon,\eta}(t)=
-\mathbf{E} \widetilde{\nabla} g_w(\mathrm{X}^{\varepsilon,\eta}(t))\nabla f_v(\mathbf{E} g_w(\mathrm{X}^{\varepsilon,\eta}(t)))dt+\sqrt{\eta}dZ_t^\varepsilon \ , \
 \mathrm{X}^{\varepsilon,\eta}(0)=x_0 \ .
\end{equation}
From \eqref{Eq:AveragedSlowMotionGD} and \eqref{Eq:AveragedSlowMotionSGD}, we know that
$$\bar{X}(t)+\sqrt{\eta}Z_t^\varepsilon-\mathrm{X}^{\varepsilon,\eta}(t)\stackrel{\mathcal{D}}{\approx} \mathcal{O}(\sqrt{\eta}) \ .$$
So that by \eqref{Eq:ExpansionNormalApproximationGD} we further have
\begin{equation}\label{Eq:ExpansionNormalApproximationSGD}
X^{\varepsilon,\eta}(t)\stackrel{\mathcal{D}}{\approx} \mathrm{X}^{\varepsilon,\eta}(t)+\mathcal{O}(\sqrt{\varepsilon})+\mathcal{O}(\sqrt{\eta}) \ .
\end{equation}

From the perspective of mathematical techniques,
there are two classical approaches to averaging principle and normal deviations
\footnote{A much more technical and functional--analytic third method is discussed in Remark 1 of Section \ref{sec:remarks}.}.
One is the classical Khasminskii's averaging method
\cite{[Khasminskii1968AveragingSDE]}. This method chooses an intermediate time scale $\Delta\rightarrow 0$ such that $\dfrac{\Delta}{\eta}\rightarrow \infty$.
This intermediate time scale enables the analysis of averaging procedure by using a fast motion with
frozen slow component. To demonstrate its effectiveness, in this work we exploit
this method to do our averaging analysis.
Another less intuitive method is the corrector method, which relies on the solution of an auxiliary Poisson equation. Upon
obtaining appropriate a--priori estimates for this Poisson equation, one can reduce the averaging principle or normal deviations
to the analysis of an It\^{o}'s formula. Since we are working in the case when fast motion $Y^{\varepsilon,\eta}(t)$ is an OU process,
when applying the corrector method,
we are mostly close to the set--up of \cite{[Pardoux-Veretennikov2]} (see also \cite{[Pardoux-Veretennikov1]}, \cite{[Pardoux-Veretennikov3]}, \cite{[CerraiNormalDeviation2009]}).
Our analysis of the normal deviations will be following the corrector method and based on a--priori bounds provided in \cite{[Pardoux-Veretennikov2]}.

\subsection{Connection with stochastic compositional gradient descent algorithm.}\label{ssec:connection}

In the field of statistical optimization, the stochastic composition optimization problem of the following form has been of tremendous interests in both theory and application:
\begin{equation}\label{Eq:OptimizationProblem}
\min\limits_{x} \left(\mathbf{E} f_v \circ \mathbf{E} g_w \right)(x).
\end{equation}
Here $x \in \mathbb{R}^q$, $f\circ g \equiv f(g(x))$ denotes the composite function, and $(v, w)$ denotes a pair of random variables.
\cite{[SCGDpaper]} has shown that the optimization problem \eqref{Eq:OptimizationProblem} includes many important applications in statistical learning and finance, such as reinforcement learning, statistical estimation, dynamic programming and portfolio management.

Let us consider the following version of Stochastic Composite Gradient Descent (SCGD) algorithm in \cite[Algorithm 1]{[SCGDpaper]} whose iteration takes the form
\begin{equation}\label{Eq:SCGDAlgorithm}
\left\{\begin{array}{ll}
y_{k+1}=(1-\varepsilon)y_k+\varepsilon g_{w_k}(x_k) \ , & y_0\in \mathbb{R}^m \ ,
\\
x_{k+1}=x_k-\eta \widetilde{\nabla} g_{w_k}(x_k)\nabla f_{v_k}(y_{k+1}) \ , & x_0\in \mathbb{R}^n \ .
\end{array}\right.
\end{equation}
Here $(w_k,v_k)$ is taken as i.i.d.~random vectors following some distribution $\mathcal{D}$ over the parameter space;
\footnote{Often in optimization for finite samples, the parameter space is chosen as some finite, discrete index set $\{1,\dots, N_1\}  \times \{1,\dots,N_2\}$, and $\mathcal{D}$ is the uniform distribution over such index set. We extend this setting to any distribution over general parameter space.}
$f_{v_k}: \mathbb{R}^m\rightarrow \mathbb{R}$ and $g_{w_k}: \mathbb{R}^n\rightarrow \mathbb{R}^m$ are functions indexed by the aformentioned random vectors; the vector $\nabla f_{v_k}(y_{k+1})$
is the gradient column $m$--vector of $f_{v_k}$ evaluated at $y_{k+1}$
and the matrix $\widetilde{\nabla} g_{w_k}(x_k)$ is the $n\times m$ matrix formed by the gradient column
$n$--vector of each of the $m$ components of $g_{w_k}$ evaluated at $x_k$.
The SCGD algorithm \eqref{Eq:SCGDAlgorithm} is a provably effective method that solves \eqref{Eq:OptimizationProblem}; see early optimization literatures on the convergence and rates of convergence analysis in \cite{[Ermoliev],[SCGDpaper]}. However, the convergence rate of SCGD algorithm and its variations is not known to be comparable to its SGD counterpart \cite{[SCGDpaper],[SCGDpaper2]}. To drill further into this algorithm we consider the coupled diffusion process \eqref{Eq:DiffusionLimit} which is a continuum version, as both $\varepsilon,\eta\to 0$ and $\varepsilon / \eta \to \infty$, of the SCGD algorithm \eqref{Eq:SCGDAlgorithm}. We copy in below the perturbed compositional gradient flow \eqref{Eq:DiffusionLimit} for convenience:
\begin{equation}\label{Eq:DiffusionLimit2}
\left\{\begin{array}{ll}
dy(t)=-\varepsilon y(t) dt +\varepsilon \mathbf{E} g_w(x(t))dt+ \varepsilon \Sigma_1(x(t))dW_t^1  \ , & y(0)=y_0 \ ,
\\
dx(t)=-\eta \mathbf{E}\widetilde{\nabla} g_w(x(t))\nabla f_{v}(y(t))dt+\eta \Sigma_2(x(t), y(t))dW_t^2 \ , & x(0)=x_0 \ .
\end{array}\right.
\end{equation}
Here $(w,v)$ is taken to be distributed as $\mathcal{D}$, and $f_v: \mathbb{R}^m\rightarrow \mathbb{R}$ and $g_w: \mathbb{R}^n\rightarrow \mathbb{R}^m$ are assumed to be in $\mathbf{C}^{(4)}$. Without loss of generality, when considering an optimization problem \eqref{Eq:OptimizationProblem}, we can assume that the functions $f$ and $g$
are supported on some compact subsets of $\mathbb{R}^m$ and $\mathbb{R}^n$, respectively\footnote{See Remark 3 of Section \ref{sec:remarks}.}. Also for convenience, let us further assume that the $w$ and $v$ in the $(w,v)$-pair drawn from $\mathcal{D}$ are independent. We do not believe this assumption is necessary, see discussions in \cite{[SCGDpaper],[SCGDpaper2]}; however it does simplify our analysis since $(W_t^1,W_t^2)$ in the perturbed compositional gradient flow \eqref{Eq:DiffusionLimit2} can be chosen as an independent pair of Brownian motions which in turn simplifies the proof.

Recall that $(X^{\varepsilon,\eta}(t), Y^{\varepsilon,\eta}(t))=(x(t/\eta), y(t/\eta))$.
In the case where the objective function $\left(\mathbf{E} f_v \circ \mathbf{E} g_w \right)(x)$ is strongly convex, $\mathrm{X}^{\varepsilon,\eta}(t)$ in \eqref{Eq:AveragedSlowMotionSGD} enters a basin containing the minimizer of
 \eqref{Eq:OptimizationProblem} in finite time $T>0$, so that
\eqref{Eq:ExpansionNormalApproximationSGD} implies $X^{\varepsilon,\eta}(t)$ in \eqref{Eq:DiffusionLimitTimeChanged} enters a basin containing the minimizer of  \eqref{Eq:OptimizationProblem}
also in finite time $T>0$. Such heuristic analysis validates, in the sense of convergence, the effectiveness of using the perturbed compositional gradient flow to solve \eqref{Eq:OptimizationProblem} in the strongly convex case. Such argument can be generalized to the convex case and omitted due to the limitation of space.

It is worth pointing out that in an early probability literature \cite{[Pardoux-Veretennikov2]}, the authors have briefly mentioned in its introductory part the potential application of averaging principle to the analysis of stochastic approximation algorithms.
In contrast, in the classical literature on stochastic approximation algorithms (see \cite{[AdaptedAlgorithmbook]}, \cite{[Borkar]},\cite{[Kushner-Yin]}), the techniques of normal deviations have been addressed under the context of weak convergence to diffusion processes in the discrete setting.
For example, \cite[Chap.~4, Part II]{[AdaptedAlgorithmbook]} analyzed the asymptotic behavior of a board class of single-equation adaptive algorithms including SGD.
Moreover, \cite[Chap.~8]{[Kushner-Yin]} discussed the idea of multiple timescale analysis for stochastic approximation algorithms; see also \cite[Chap.~6]{[Borkar]} for a connection to averaging principle for constant stepsize algorithms. However, these mathematical theories focus on the long-time asymptotic analysis instead of convergence rates, which is vital in many recent applications.
The current work serves as an attempt on convergence rates using one algorithmic example (SCGD) and can be viewed as a further contribution along this line of research thread.

\vspace*{4pt}

\noindent\textbf{Organization.}
The paper is organized as follows.
In Section \ref{sec:averaging} we will show the averaging principle that justifies the convergence of $X^{\varepsilon,\eta}(t)$
to $X^\varepsilon(t)$ as $\eta \rightarrow 0$. In Section \ref{sec:normal} we will consider the rescaled deviation
$Z_t^{\varepsilon,\eta}=(X^{\varepsilon,\eta}(t)-X^{\varepsilon}(t)) / \sqrt{\eta}$ and we show that as $\eta \rightarrow 0$ it converges weakly to the process $Z_t^\varepsilon$. This justifies \eqref{Eq:ExpansionNormalApproximation}.
In Section \ref{sec:error} we show the approximation \eqref{Eq:ExpansionNormalApproximationSGD} and we justify the effectiveness of using SCGD in the strongly convex case.
In Section \ref{sec:remarks} we discuss further problems, remarks and generalizations.

\vspace*{4pt}

\noindent\textbf{Notational Conventions.} For an $n$--vector $v=(v_1,...,v_n)$
we define the norm
$$|v|_{\mathbb{R}^n}=(v_1^2+...+v_n^2)^{1/2} \ .$$
We also denote $[v]_k=v_k$ for $k=1,2,...,n$.
For any $n\times n$ matrix $\sigma\in \mathbb{R}^n \otimes \mathbb{R}^n$, let us define the norm
$$\|\sigma\|_{\mathbb{R}^n\otimes \mathbb{R}^n}=\left(\sum\limits_{i,j=1}^n \sigma_{ij}^2\right)^{1/2} \ .$$
If $q$ is a vector or a matrix, then
$|q|_{\text{norm}}$ denotes either $|q|_{\mathbb{R}^n}$ when $q$ is an $n$--vector, or $\|q\|_{\mathbb{R}^n\otimes \mathbb{R}^n}$ if $q$
is an $n\times n$ matrix. The standard inner product in $\mathbb{R}^n$ is denoted as $\langle \bullet, \bullet \rangle_{\mathbb{R}^n}$.

The spaces $\mathbf{C}^{(i)}(D)$, $i=0,1,...$ (and $\mathbf{C}(D)=\mathbf{C}^{(0)}(D)$) are the spaces of
$i$--times continuously differentiable functions on a domain $D$ ($D$ can be the whole space). For a
function $f\in \mathbf{C}^{(i)}(D)$
we define $\|f\|_{i}$ to be the $\mathbf{C}^{(i)}(D)$ norm of $f$ on $D$.
In case we need to highlight the target space, we also use $\mathbf{C}^{(i)}(D;M)$ that refers to functions
in the space $\mathbf{C}^{(i)}(D)$ that are mapped into $M$.
If $f\in \text{Lip}(D)$ is Lipschitz continuous on $D$, then
$[f]_{\text{Lip}}$ is the Lipschitz seminorm $[f]_{\text{Lip}}=\sup\limits_{x,y\in D}\dfrac{|f(x)-f(y)|_{\text{norm}}}{|x-y|_{\text{norm}}}$.
In the case of vector or matrix valued functions, the Lipschitz norm is then defined to be the largest Lipschitz norm for its
corresponding component functions.

Throughout the paper, capital $X(t), Y(t), \bar{X}(t)$, etc., are quantities
for the time rescaled process \eqref{Eq:DiffusionLimitTimeChanged},
and small $x(t), y(t), \bar{x}(t)$, etc., are quantities for the original process \eqref{Eq:DiffusionLimit}.
The constant $C$ denotes a positive constant that varies from line to line. Sometimes, to emphasize
the dependence of this constant on other parameters, $C=C(\bullet)$ may also be used.
For notational convenience, we use simultaneously, e.g., $X(t)$ or $X_t$ to denote a stochastic process.

\section{The convergence of $X^{\varepsilon,\eta}(t)$ to $X^\varepsilon(t)$: Averaging principle.}\label{sec:averaging}

In this section we are going to show the convergence of $X^{\varepsilon,\eta}(t)$ to $X^{\varepsilon}(t)$
as $\eta \rightarrow 0$ by arguing as in the classical averaging principle (see \cite{[Khasminskii1966SmallParameterDE]}, \cite{[Khasminskii1968AveragingSDE]}).

Our first Lemma is about $L^2$--boundedness of the system $(X^{\varepsilon,\eta}_t, Y^{\varepsilon,\eta}_t)$ in \eqref{Eq:DiffusionLimitTimeChangedStandardForm}.

\begin{lemma}\label{Lemma:L2BoundedNessSlowFastProcessTimeChanged}
For any $T>0$ and $0<\eta<1$ there exist some constant $C=C(T,\varepsilon)>0$ such that
\begin{equation}\label{Lemma:L2BoundedNessSlowFastProcessTimeChanged:Eq:XProcessBound}
\sup\limits_{0\leq t\leq T}\mathbf{E}|X^{\varepsilon,\eta}_t|_{\mathbb{R}^n}^2\leq C(1+|x_0|_{\mathbb{R}^n}^2) \ ,
\end{equation}
and
\begin{equation}\label{Lemma:L2BoundedNessSlowFastProcessTimeChanged:Eq:YProcessBound}
\sup\limits_{0\leq t\leq T}\mathbf{E}|Y^{\varepsilon,\eta}_t|_{\mathbb{R}^n}^2\leq C(1+|y_0|_{\mathbb{R}^m}^2) \ .
\end{equation}
\end{lemma}

\begin{proof}
This Lemma can be derived in the same way as in \cite[Lemma 4.2]{[CerraiKhasminskii2009]}. In fact, we can write the equation
\eqref{Eq:DiffusionLimitTimeChangedStandardForm} for $X^{\varepsilon,\eta}_t$ in an integral form as
$$
X^{\varepsilon,\eta}_t=x_0+\int_0^t B_2(X^{\varepsilon,\eta}_s, Y^{\varepsilon,\eta}_s)ds+\sqrt{\eta}\int_0^t
\Sigma_2(X^{\varepsilon,\eta}_s, Y^{\varepsilon,\eta}_s)dW_s^2 \ .
$$

Therefore
$$\begin{array}{ll}
& \mathbf{E}|X^{\varepsilon,\eta}_t|_{\mathbb{R}^n}^2
\\
\leq &\displaystyle{ C\left(|x_0|_{\mathbb{R}^n}^2+\mathbf{E}\left|\int_0^t B_2(X^{\varepsilon,\eta}_s, Y^{\varepsilon,\eta}_s)ds\right|_{\mathbb{R}^n}^2
+\eta\mathbf{E}\left|\int_0^t
\Sigma_2(X^{\varepsilon,\eta}_s, Y^{\varepsilon,\eta}_s)dW_s^2\right|_{\mathbb{R}^n}^2\right)  \ .}
\\
\end{array}$$

For a matrix valued random function $\sigma(t)=\sigma(\omega,t)$ adapted to the filtration of $W_t$
we have (see \cite[(3.12) and (3.13)]{[Kifer1981SaddlePoint]})
\begin{equation}\label{Eq:KiferMatrixStochasticIntegralItoIsometry}
\mathbf{E} \left|\int_0^t \sigma(t)dW_t\right|_{\mathbb{R}^n}^2=\int_0^t \mathbf{E} \|\sigma(t)\|_{\mathbb{R}^n\otimes \mathbb{R}^n}^2 dt \ .
\end{equation}

Therefore we obtain \eqref{Lemma:L2BoundedNessSlowFastProcessTimeChanged:Eq:XProcessBound}.

We can write the solution $Y_t^{\varepsilon,\eta}$ in \eqref{Eq:DiffusionLimitTimeChangedStandardForm} in mild form as
\begin{equation}\label{Eq:MildFormYprocess}
Y_t^{\varepsilon,\eta}=e^{-\frac{\varepsilon}{\eta}t}y_0+\dfrac{\varepsilon}{\eta}\int_0^t e^{-\frac{\varepsilon}{\eta}(t-s)} B_1(X_s^{\varepsilon,\eta})ds
+\dfrac{\varepsilon}{\sqrt{\eta}}\int_0^t e^{-\frac{\varepsilon}{\eta}(t-s)}\Sigma_1(X_s^{\varepsilon,\eta})dW_s^1 \ .
\end{equation}

Set $\Gamma(t)=\displaystyle{\dfrac{\varepsilon}{\sqrt{\eta}}\int_0^t e^{-\frac{\varepsilon}{\eta}(t-s)}\Sigma_1(X_s^{\varepsilon,\eta})dW_s^1}$ and $\Lambda(t)=Y^{\varepsilon,\eta}_t-\Gamma(t)$.
Then we have
$$d\Lambda(t)=-\dfrac{\varepsilon}{\eta}[\Lambda(t)+B_1(X^{\varepsilon,\eta}_t)]dt \ , \ \Lambda(0)=y_0 \ ,$$
which gives
$$
\dfrac{1}{2}\dfrac{d}{dt}|\Lambda(t)|_{\mathbb{R}^n}^2  =\left\langle \Lambda(t), -\dfrac{\varepsilon}{\eta}[\Lambda(t)+B_1(X^{\varepsilon,\eta}_t)]\right\rangle_{\mathbb{R}^n} \leq  -\dfrac{\varepsilon}{2\eta}|\Lambda(t)|_{\mathbb{R}^n}^2+\dfrac{C\varepsilon}{\eta} \ .
$$

Therefore by Gronwall inequality we know that for $0\leq t\leq T$ we have
$$|\Lambda(t)|_{\mathbb{R}^n}^2\leq Ce^{-\frac{\varepsilon}{\eta}t}|y_0|_{\mathbb{R}^m}^2+2CT\leq C(1+|y_0|_{\mathbb{R}^m}^2) \ .$$

It remains to estimate $\mathbf{E}|\Gamma(t)|_{\mathbb{R}^m}^2$. Again, by \eqref{Eq:KiferMatrixStochasticIntegralItoIsometry} we have
$$
\mathbf{E}|\Gamma(t)|_{\mathbb{R}^m}^2 = \dfrac{\varepsilon^2}{\eta} e^{-\frac{2\varepsilon}{\eta}t}\displaystyle{\int_0^t e^{\frac{2\varepsilon}{\eta}s}\|\Sigma_1(X^{\varepsilon,\eta}_s)\|_{\mathbb{R}^n\otimes \mathbb{R}^n}^2ds}
\leq \dfrac{C \varepsilon}{2} \ .$$

Thus we obtain
$$\mathbf{E}|Y^{\varepsilon,\eta}_t|_{\mathbb{R}^m}^2 \leq C(\mathbf{E}|\Lambda(t)|_{\mathbb{R}^m}^2+\mathbf{E}|\Gamma(t)|_{\mathbb{R}^m}^2)\leq C(1+|y_0|_{\mathbb{R}^m}^2) \ ,$$
which is \eqref{Lemma:L2BoundedNessSlowFastProcessTimeChanged:Eq:YProcessBound}.
\end{proof}

The next Lemma summarizes basic facts about the process $\mathfrak{y}^{X,\varepsilon}$ defined
in \eqref{Eq:FastMotionDeterministicSystemRandomPerturbation}.

\begin{lemma}\label{Lemma:BasicFactsFastMotionDeterministicSystemRandomPerturbation}
Let the process $\mathfrak{y}^{X,\varepsilon}(t)$ defined
in \eqref{Eq:FastMotionDeterministicSystemRandomPerturbation} start from $\mathfrak{y}^{X,\varepsilon}(0)=y\in \mathbb{R}^m$.
Then for any function $\varphi: \mathbb{R}^m\rightarrow \mathbb{R}$, for some $\delta>0$ we have
\begin{equation}\label{Lemma:BasicFactsFastMotionDeterministicSystemRandomPerturbation:Eq:SemigroupExponentialConvergence}
\displaystyle{\left|\mathbf{E}_{y}\varphi(\mathfrak{y}^{X,\varepsilon}(t))-\int_{\mathbb{R}^m}\varphi(Y)\mu^{X,\varepsilon}(dY)\right|
\leq Ce^{-\delta t}(1+|y|_{\mathbb{R}^m})[\varphi]_{\text{Lip}} \ ,}
\end{equation}
where the constant $C>0$ may depend on $\varepsilon$, but is independent of $X$.

Moreover, for some constant $C>0$ we have
\begin{equation}\label{Lemma:BasicFactsFastMotionDeterministicSystemRandomPerturbation:Eq:LLNDriftSlow}
\begin{array}{l}
\displaystyle{\mathbf{E}\left|\dfrac{1}{T}\int_{t}^{t+T} B_2(X, \mathfrak{y}^{X,\varepsilon}(s))ds -
\int_{\mathbb{R}^m}B_2(X,Y)\mu^{X,\varepsilon}(dY)\right|^2_{\mathbb{R}^n}}
\\
\displaystyle{\qquad \qquad \qquad \leq \dfrac{C}{T}\left[\sqrt{\varepsilon}[B_2]_{\text{Lip}}(1+|y|_{\mathbb{R}^m})
+|B_2(X,0)|_{\mathbb{R}^n}\right]^2
\ .}
\end{array}
\end{equation}

Finally
\begin{equation}\label{Lemma:BasicFactsFastMotionDeterministicSystemRandomPerturbation:Eq:L2BoundedNessFastMotionGeneratedByOU}
\mathbf{E}|\mathfrak{y}^{X,\varepsilon}(t)|_{\mathbb{R}^m}^2\leq C(1+e^{-2t}|y|_{\mathbb{R}^m}^2) \ ,
\end{equation}
and
\begin{equation}\label{Lemma:BasicFactsFastMotionDeterministicSystemRandomPerturbation:Eq:IntegrationFastMotionAgainstInvariantMeasure}
\displaystyle{\int_{\mathbb{R}^m}|y|_{\mathbb{R}^m}^2\mu^{X, \varepsilon}(dY)\leq C<\infty  \ ,}
\end{equation}
in which the constant $C$ may depend on $\varepsilon$ but is independent of $X$.
\end{lemma}

\begin{proof}
Let us first recall the auxiliary process $\mathfrak{y}^{X,\varepsilon}(t)$ in
\eqref{Eq:FastMotionDeterministicSystemRandomPerturbation}.
From \eqref{Eq:FastMotionDeterministicSystemRandomPerturbation} we have
$$
d(\mathfrak{y}^{X,\varepsilon}(t)-\mathbf{E} g_w(X))=-(\mathfrak{y}^{X,\varepsilon}(t)- \mathbf{E}
g_w(X))dt+\sqrt{\varepsilon}\Sigma_1(X)dW_t^1 \ ,
$$
so that
$$
\mathfrak{y}^{X,\varepsilon}(t)-\mathbf{E} g_w(X)=(\mathfrak{y}^{X,\varepsilon}(0)-\mathbf{E}
g_w(X))e^{-t}+\sqrt{\varepsilon}\Sigma_1(X)\int_0^t e^{-(t-s)}dW_s^1 \ .
$$

Let $Z(t)=\displaystyle{\int_0^t e^{-(t-s)}dW_s^1}$ be the OU process satisfying the stochastic differential equation
$$dZ(t)=-Z(t)dt+dW_t^1 \ , \ Z(0)=0\in \mathbb{R}^m \ .$$
Thus we have the explicit representation
\begin{equation}\label{Lemma:BasicFactsFastMotionDeterministicSystemRandomPerturbation:Eq:FastMotionDeterministicSystemRandomPerturbationExplicitSolutionOU}
\mathfrak{y}^{X,\varepsilon}(t)=\mathbf{E} g_w(X)+(\mathfrak{y}^{X,\varepsilon}(0)-\mathbf{E}
g_w(X))e^{-t}+\sqrt{\varepsilon}\Sigma_1(X)Z(t) \ .
\end{equation}

Let $\mu(dY)\sim \mathcal{N}\left(0,\dfrac{1}{2}I_m\right)$ be the invariant measure of $Z(t)$, where $I_m$ is the identity matrix
in $\mathbb{R}^m$. Then we have the exponential mixing estimate, that for $\delta>0$ we have
\begin{equation}\label{Lemma:BasicFactsFastMotionDeterministicSystemRandomPerturbation:Eq:ExponentialMixingOU}
\left|\mathbf{E}\varphi(Z(t))-\int_{\mathbb{R}^m}\varphi(Y)\mu(dY)\right|\leq Ce^{-\delta t}[\varphi]_{\text{Lip}(\mathbb{R}^m)} \ .
\end{equation}

This, together with \eqref{Lemma:BasicFactsFastMotionDeterministicSystemRandomPerturbation:Eq:FastMotionDeterministicSystemRandomPerturbationExplicitSolutionOU},
as well as the boundedness of $\mathbf{E} g_w(X)$ in terms of $X$,
imply \eqref{Lemma:BasicFactsFastMotionDeterministicSystemRandomPerturbation:Eq:SemigroupExponentialConvergence}.

From \eqref{Lemma:BasicFactsFastMotionDeterministicSystemRandomPerturbation:Eq:ExponentialMixingOU},
by the same argument as in \cite[Lemma 2.3]{[CerraiKhasminskii2009]}, we obtain
\begin{align}\label{Lemma:BasicFactsFastMotionDeterministicSystemRandomPerturbation:Eq:LLNOUDriftSlow}
& \mathbf{E}\left|\dfrac{1}{T}\int_{t}^{t+T} B_2(X, Z(s))ds -\int_{\mathbb{R}^m}B_2(X,Y)\mu(dY)\right|^2_{\mathbb{R}^n}
\nonumber\\ \leq&
\dfrac{C}{T}\left([B_2]_{\text{Lip}}
+|B_2(X,0)|_{\mathbb{R}^n}\right)^2
\ .
\end{align}

From here, by making use of the representation \eqref{Lemma:BasicFactsFastMotionDeterministicSystemRandomPerturbation:Eq:FastMotionDeterministicSystemRandomPerturbationExplicitSolutionOU},
we obtain \eqref{Lemma:BasicFactsFastMotionDeterministicSystemRandomPerturbation:Eq:LLNDriftSlow}.

Moreover, by \eqref{Lemma:BasicFactsFastMotionDeterministicSystemRandomPerturbation:Eq:FastMotionDeterministicSystemRandomPerturbationExplicitSolutionOU}
we infer that
$$\begin{array}{ll}
 \mathbf{E}|\mathfrak{y}^{X,\varepsilon}(t)|_{\mathbb{R}^m}^2
 & \leq C(1+e^{-2t}+e^{-2t}|\mathfrak{y}^{X,\varepsilon}(0)|_{\mathbb{R}^n}^2)
+C\varepsilon\mathbf{E}|Z(t)|_{\mathbb{R}^m}^2
\\
& =C(1+e^{-2t})+\dfrac{C\varepsilon}{2}(1-e^{-2t})
\\
& \leq C(1+e^{-2t}|\mathfrak{y}^{X,\varepsilon}(0)|_{\mathbb{R}^n}^2) \ ,
\end{array}$$
which is \eqref{Lemma:BasicFactsFastMotionDeterministicSystemRandomPerturbation:Eq:L2BoundedNessFastMotionGeneratedByOU}.

Finally, \eqref{Lemma:BasicFactsFastMotionDeterministicSystemRandomPerturbation:Eq:IntegrationFastMotionAgainstInvariantMeasure}
 is a result of \eqref{Eq:InvariantMeasureOUSmallDiffusion} and the fact that $\mathbf{E} g_w(X)$ and $\Sigma_1(X)$ are uniformly bounded in $X$.
\end{proof}

Now we will derive the averaging principle
following the classical method in \cite{[Khasminskii1968AveragingSDE]}.
Let $T>0$. Let us consider a partition of the time interval $[0,T]$ into intervals of the same length $\Delta>0$. Let us introduce the auxiliary processes
$\widehat{Y}^{\varepsilon,\eta}_t$, $\widehat{X}^{\varepsilon.\eta}_t$ by means of the relations
\begin{equation}\label{Eq:DiffusionLimitTimeChanged:AuxiliaryProcessYhat}
\begin{array}{l}
\displaystyle{\widehat{Y}^{\varepsilon,\eta}_t=Y^{\varepsilon,\eta}_{k\Delta}-\dfrac{\varepsilon}{\eta}\int_{k\Delta}^t \widehat{Y}^{\varepsilon,\eta}_sds
+\dfrac{\varepsilon}{\eta}\int_{k\Delta}^t B_1(X_{k\Delta}^{\varepsilon,\eta})ds+\dfrac{\varepsilon}{\sqrt{\eta}}\int_{k\Delta}^t \Sigma_1(X^{\varepsilon,\eta}_{k\Delta})dW_s^1 \ , }
\\
\qquad \qquad \qquad \ t\in [k\Delta, (k+1)\Delta] \ ,
\end{array}
\end{equation}
\begin{equation}\label{Eq:DiffusionLimitTimeChanged:AuxiliaryProcessXhat}
\widehat{X}^{\varepsilon,\eta}_t=x_0+\int_0^t B_2(X^{\varepsilon,\eta}_{[s/\Delta]\Delta}, \widehat{Y}^{\varepsilon,\eta}_s)ds+\sqrt{\eta}\int_0^t
\Sigma_2(X^{\varepsilon,\eta}_{[s/\Delta]\Delta}, \widehat{Y}^{\varepsilon,\eta}_s)dW_s^2 \ .
\end{equation}

\begin{lemma}\label{Lemma:L2ClosessYandYhat}
The interval length $\Delta=\Delta(\eta)$ can be chosen such that $\eta^{-1}\Delta(\eta)\rightarrow \infty$, $\Delta(\eta)\rightarrow 0$ as $\eta \rightarrow 0$ and
for any small $0<\kappa<1$ we have
\begin{equation}\label{Lemma:L2ClosessYandYhat:Eq:L2ClosessYandYhat}
\mathbf{E} |Y_t^{\varepsilon,\eta}-\widehat{Y}_t^{\varepsilon,\eta}|_{\mathbb{R}^m}^2\leq C\varepsilon^2\eta^{1-\kappa}\rightarrow 0
\end{equation}
uniformly in $x_0\in \mathbb{R}^n$, $y_0\in \mathbb{R}^m$ and $t\in [0,T]$.
\end{lemma}

\begin{proof}

In fact we can write, for $t\in [k\Delta, (k+1)\Delta]$, that
$$\begin{array}{ll}
& \mathbf{E} |Y_t^{\varepsilon,\eta}-\widehat{Y}_t^{\varepsilon,\eta}|_{\mathbb{R}^m}^2
\\
= & \displaystyle{\mathbf{E}\left|
-\dfrac{\varepsilon}{\eta}\int_{k\Delta}^t (Y_s^{\varepsilon,\eta}-\widehat{Y}_s^{\varepsilon,\eta})ds+\dfrac{\varepsilon}{\eta}\int_{k\Delta}^t [B_1(X^{\varepsilon,\eta}_s)-B_1(X^{\varepsilon,\eta}_{k\Delta})]ds\right.}
\\
& \displaystyle{\left.\qquad \qquad +
\dfrac{\varepsilon}{\sqrt{\eta}}\int_{k\Delta}^t [\Sigma_1(X_s^{\varepsilon,\eta})-\Sigma_1(X_{k\Delta}^{\varepsilon,\eta})]dW_s^1\right|_{\mathbb{R}^m}^2}
\\
= & \displaystyle{C\dfrac{\varepsilon^2}{\eta^2}\Delta\int_{k\Delta}^t \mathbf{E}|Y^{\varepsilon,\eta}_s-\widehat{Y}^{\varepsilon,\eta}_s|_{\mathbb{R}^m}^2ds
+C\dfrac{\varepsilon^2}{\eta^2}\Delta\int_{k\Delta}^t \mathbf{E}|X^{\varepsilon,\eta}_s-X^{\varepsilon,\eta}_{k\Delta}|_{\mathbb{R}^n}^2ds}
\\
& \displaystyle{\qquad \qquad +
C\dfrac{\varepsilon^2}{\eta}\int_{k\Delta}^t \mathbf{E}|X^{\varepsilon,\eta}_s-X^{\varepsilon,\eta}_{k\Delta}|_{\mathbb{R}^n}^2ds \ .}
\end{array}$$

It follows from the boundedness of the coefficients of the stochastic equation \eqref{Eq:DiffusionLimitTimeChangedStandardForm}
for $X^{\varepsilon,\eta}$ that we have
\begin{equation}\label{Lemma:L2ClosessYandYhat:Eq:ClosenessDiscretizedSlowMotion}
\mathbf{E}|X_s^{\varepsilon,\eta}-X_{k\Delta}^{\varepsilon,\eta}|_{\mathbb{R}^n}^2\leq C|s-k\Delta|^2\leq C\Delta^2 \ ,
\end{equation}
for $0<\Delta<1$ and $s\in [k\Delta, (k+1)\Delta]$. Therefore we have
$$\mathbf{E} |Y_t^{\varepsilon,\eta}-\widehat{Y}_t^{\varepsilon,\eta}|_{\mathbb{R}^m}^2
\leq \displaystyle{C\varepsilon^2\left(\dfrac{\Delta}{\eta^2}+\dfrac{1}{\eta}\right)\Delta^2
+C\varepsilon^2\dfrac{\Delta}{\eta^2}\int_{k\Delta}^t \mathbf{E}|Y^{\varepsilon,\eta}_s-\widehat{Y}^{\varepsilon,\eta}_s|_{\mathbb{R}^m}^2ds
 \ .}
$$
By Gronwall's inequality this implies that we have, for each $k$ and every $t\in [k\Delta, (k+1)\Delta]$, that
$$\mathbf{E} |Y_t^{\varepsilon,\eta}-\widehat{Y}_t^{\varepsilon,\eta}|_{\mathbb{R}^m}^2
\leq \displaystyle{C\varepsilon^2\left(\dfrac{\Delta}{\eta^2}+\dfrac{1}{\eta}\right)\Delta^2
\exp\left(C\varepsilon^2\dfrac{\Delta^2}{\eta^2}\right)
 \ .}
$$

We then pick
\begin{equation}\label{Lemma:L2ClosessYandYhat:Eq:ChoiceOfDelta}
\Delta=\Delta(\eta)=\eta\sqrt[4]{\ln(\eta^{-1})}
\end{equation}
and we conclude \eqref{Lemma:L2ClosessYandYhat:Eq:L2ClosessYandYhat}
by making use of the asymptotic that for any small $a>0$ fixed,
we have $\dfrac{\ln \eta^{-1}}{\eta^{-a}}\rightarrow 0$, $\dfrac{\sqrt{\ln(\eta^{-1})}}{\ln(\eta^{-a})}\rightarrow 0$ as $\eta \rightarrow 0$.
In fact, we have
$$\begin{array}{ll}
C\varepsilon^2\left(\dfrac{\Delta}{\eta^2}+\dfrac{1}{\eta}\right)\Delta^2
\exp\left(C\varepsilon^2\dfrac{\Delta^2}{\eta^2}\right)
 & \leq C\varepsilon^2 (\eta^{-\frac{a}{2}}+\eta^{-\frac{a}{4}})\eta\cdot \eta^{-\frac{a}{4}}\exp(C\varepsilon^2 \ln (\eta^{-a}))
 \\
& \leq C\varepsilon^2 (\eta^{-\frac{a}{2}}+\eta^{-\frac{a}{4}})\eta\cdot \eta^{-\frac{a}{4}}\eta^{-aC\varepsilon^2} \ ,
\end{array}$$
and thus we can pick $\kappa=\left(\dfrac{3}{4}+C\varepsilon^2\right)a>0$ to be any small positive number.
\end{proof}

\begin{lemma}\label{Lemma:ClosenessXToXhat}
For any small $0<\kappa<1$ we have
\begin{equation}\label{Lemma:ClosenessXToXhat:Eq:ClosenessXToXhat}
\sup\limits_{0\leq t\leq T}\mathbf{E}|X^{\varepsilon,\eta}_t-\widehat{X}^{\varepsilon,\eta}_t|_{\mathbb{R}^n}^2\leq C(T^2+T^2\varepsilon^2+1)\eta^{1-\kappa} \rightarrow 0 \ .
\end{equation}
\end{lemma}

\begin{proof}

By \eqref{Eq:DiffusionLimitTimeChangedStandardForm} we know that we can write the process $X_t^{\varepsilon,\eta}$ as an integral equation
\begin{equation}\label{Lemma:ClosenessXToXhat:Eq:DiffusionLimitTimeChanged:SlowProcessX}
X^{\varepsilon,\eta}_t=x_0+\int_0^t B_2(X^{\varepsilon,\eta}_s, Y^{\varepsilon,\eta}_s)ds+\sqrt{\eta}\int_0^t
\Sigma_2(X^{\varepsilon,\eta}_s, Y^{\varepsilon,\eta}_s)dW_s^2 \ .
\end{equation}

Comparing \eqref{Eq:DiffusionLimitTimeChanged:AuxiliaryProcessXhat} and
\eqref{Lemma:ClosenessXToXhat:Eq:DiffusionLimitTimeChanged:SlowProcessX} we see that
$$\begin{array}{ll}
&\mathbf{E}|X^{\varepsilon,\eta}_t-\widehat{X}^{\varepsilon,\eta}_t|_{\mathbb{R}^n}^2
\\
\leq & \displaystyle{C\mathbf{E}\left|\int_0^t [B_2(X_s^{\varepsilon,\eta},Y_s^{\varepsilon,\eta})-B_2(X^{\varepsilon,\eta}_{[s/\Delta]\Delta}, \widehat{Y}_s^{\varepsilon,\eta})]ds\right|_{\mathbb{R}^n}^2}
\\
&\displaystyle{\qquad \qquad \qquad +\eta C\mathbf{E}\left|\int_0^t [\Sigma_2(X_s^{\varepsilon,\eta}, Y_s^{\varepsilon,\eta})-
\Sigma_2(X_{[s/\Delta]\Delta}^{\varepsilon,\eta}, \widehat{Y}_s^{\varepsilon,\eta})]dW_s^2\right|_{\mathbb{R}^n}^2} \ ,
\end{array}$$
and so that, by \eqref{Lemma:L2ClosessYandYhat:Eq:L2ClosessYandYhat}, the Cauchy--Schwarz inequality and Lipschitz continuity of $B_2(X,Y)$
with respect to $X$, $Y$ we know that
$$\begin{array}{ll}
&\sup\limits_{0\leq t\leq T}\mathbf{E}|X^{\varepsilon,\eta}_t-\widehat{X}^{\varepsilon,\eta}_t|_{\mathbb{R}^n}^2
\\
\leq & \displaystyle{CT\sum\limits_{k=0}^{[T/\Delta]} \int_{k\Delta}^{(k+1)\Delta}\left(\mathbf{E}|X_s^{\varepsilon,\eta}-X^{\varepsilon,\eta}_{k\Delta}|_{\mathbb{R}^n}^2+\mathbf{E}|Y_s^{\varepsilon,\eta}-\widehat{Y}_s^{\varepsilon,\eta}|_{\mathbb{R}^m}^2\right)ds+C\eta+C\Delta}
\\
\leq & \displaystyle{CT^2\Delta^2+CT\int_0^T \mathbf{E}|Y_s^{\varepsilon,\eta}-\widehat{Y}_s^{\varepsilon,\eta}|_{\mathbb{R}^m}^2 ds+C\eta+C\Delta}
\\
\leq & \displaystyle{CT^2\Delta^2+CT^2\varepsilon^2\eta^{1-\kappa}+C\eta+C\Delta \ .}
\end{array}$$
From the asymptotic that for any small $a>0$ fixed,
we have $\dfrac{\ln \eta^{-1}}{\eta^{-a}}\rightarrow 0$, we have $\Delta\leq \eta^{1-\frac{a}{4}}$.
This implies that
$$\begin{array}{ll}
CT^2\Delta^2+CT^2\varepsilon^2\eta^{1-\kappa}+C\eta+C\Delta & \leq CT^2\eta^{2-\frac{a}{2}}+CT^2\varepsilon^2\eta^{1-\kappa}+C\eta+C\eta^{1-\frac{a}{4}}
\\
& \leq C(T^2+T^2\varepsilon^2+1)\eta^{(2-\frac{a}{2})\wedge(1-\kappa)\wedge(1-\frac{a}{4})} \ .
\end{array}$$
Thus for possibly another small $\kappa>0$
$$
\sup\limits_{0\leq t\leq T}\mathbf{E}|X^{\varepsilon,\eta}_t-\widehat{X}^{\varepsilon,\eta}_t|_{\mathbb{R}^n}^2\leq C(T^2+T^2\varepsilon^2+1)\eta^{1-\kappa} \rightarrow 0
$$
as $\varepsilon\rightarrow 0$.
\end{proof}

\begin{proposition}\label{Proposition:ConvergenceXEpsEtaToXEps}
For any $T>0$ and $\varepsilon>0,\eta>0$ small enough, for $0\leq t\leq T$ and any small $0<\kappa<1$ we have
\begin{equation}\label{Proposition:ConvergenceXEpsEtaToXEps:Eq:ClosenessXEpsEtaToXEps}
\mathbf{E}|X^{\varepsilon,\eta}_t-X^\varepsilon_t|_{\mathbb{R}^n}^2\leq \dfrac{C}{\varepsilon}\dfrac{1}{\sqrt[4]{\ln(\eta^{-1})}} \ ,
\end{equation}
for some constant $C=C(T)>0$.
\end{proposition}

\begin{proof}

By the defining equation \eqref{Eq:DiffusionLimitTimeChanged:AuxiliaryProcessYhat} of the process
$\widehat{Y}^{\varepsilon,\eta}_t$ we know that we have $\widehat{Y}^{\varepsilon,\eta}_t=\mathfrak{y}^{X_{k\Delta}^{\varepsilon,\eta}, \varepsilon}_{t\varepsilon/\eta}$
with $\mathfrak{y}^{X_{k\Delta}^{\varepsilon,\eta}, \varepsilon}_{0}=Y_{k\Delta}^{\varepsilon,\eta}$.
This, together with Lemma \ref{Lemma:BasicFactsFastMotionDeterministicSystemRandomPerturbation} estimate
\eqref{Lemma:BasicFactsFastMotionDeterministicSystemRandomPerturbation:Eq:LLNDriftSlow} imply that
\begin{equation}\label{Proposition:ConvergenceXEpsEtaToXEps:Eq:ConvergenceMixingPartFastMotion}
\begin{array}{l}
\displaystyle{\dfrac{1}{\Delta^2}\mathbf{E}\left|\int_{k\Delta}^{(k+1)\Delta}
[B_2(X^{\varepsilon,\eta}_{k\Delta}, \widehat{Y}_s^{\varepsilon,\eta})-\overline{B_2(X^{\varepsilon,\eta}_{k\Delta},Y)}^\varepsilon(X^{\varepsilon,\eta}_{k\Delta})]ds\right|^2_{\mathbb{R}^n}}
\\
\displaystyle{\qquad \qquad \leq \dfrac{C}{\dfrac{\varepsilon}{\eta}\Delta}(1+\varepsilon+\varepsilon\mathbf{E}|X_{k\Delta}^{\varepsilon,\eta}|^2_{\mathbb{R}^n}+\varepsilon\mathbf{E}|Y_{k\Delta}^{\varepsilon,\eta}|^2_{\mathbb{R}^m}) \ .}
\end{array}
\end{equation}

By making use of \eqref{Lemma:L2ClosessYandYhat:Eq:L2ClosessYandYhat},
\eqref{Lemma:L2ClosessYandYhat:Eq:ClosenessDiscretizedSlowMotion}, \eqref{Lemma:ClosenessXToXhat:Eq:ClosenessXToXhat},
\eqref{Proposition:ConvergenceXEpsEtaToXEps:Eq:ConvergenceMixingPartFastMotion},
\eqref{Lemma:RegularityBarEpsilonOperator:Eq:LipschitzContinuityBarEpsOperator} we have
{\small$$\begin{array}{ll}
&\mathbf{E}|\widehat{X}^{\varepsilon,\eta}_t-X^\varepsilon_t|_{\mathbb{R}^n}^2
\\
=& \displaystyle{\mathbf{E}\left|\int_0^t [B_2(X^{\varepsilon,\eta}_{[s/\Delta]\Delta}, \widehat{Y}_s^{\varepsilon,\eta})-\overline{B_2(X^\varepsilon_s,Y)}^\varepsilon(X^\varepsilon_s)]ds
+\sqrt{\eta}\int_0^t \Sigma_2(\widehat{X}_s^{\varepsilon,\eta}, \widehat{Y}_s^{\varepsilon,\eta})dW_s^2\right|^2_{\mathbb{R}^n}}
\\
\leq & \displaystyle{C\mathbf{E}\left|\sum\limits_{k=0}^{[t/\Delta]}\int_{k\Delta}^{(k+1)\Delta}
[B_2(X^{\varepsilon,\eta}_{k\Delta}, \widehat{Y}_s^{\varepsilon,\eta})-\overline{B_2(X^{\varepsilon,\eta}_{k\Delta},Y)}^\varepsilon(X^{\varepsilon,\eta}_{k\Delta})]ds\right|^2_{\mathbb{R}^n}}
\\
& \qquad \qquad \displaystyle{+ C\mathbf{E}\left|\int_{0}^{t}
[\overline{B_2(X^{\varepsilon,\eta}_{[s/\Delta]\Delta},Y)}^\varepsilon(X^{\varepsilon,\eta}_{[s/\Delta]\Delta})-\overline{B_2(X^{\varepsilon,\eta}_s,Y)}^\varepsilon(X^{\varepsilon,\eta}_s)]ds\right|^2_{\mathbb{R}^n}}
\\
& \qquad \qquad \displaystyle{+ C\mathbf{E}\left|\int_{0}^{t}
[\overline{B_2(X^{\varepsilon,\eta}_{s},Y)}^\varepsilon(X^{\varepsilon,\eta}_{s})-\overline{B_2(\widehat{X}^{\varepsilon,\eta}_{s},Y)}^\varepsilon(\widehat{X}^{\varepsilon,\eta}_{s})]ds\right|^2_{\mathbb{R}^n}}
\end{array}$$$$\begin{array}{ll}& \qquad \qquad \displaystyle{+ C\mathbf{E}\left|\int_{0}^{t}
[\overline{B_2(\widehat{X}^{\varepsilon,\eta}_{s},Y)}^\varepsilon(\widehat{X}^{\varepsilon,\eta}_{s})-\overline{B_2(X^\varepsilon_s,Y)}^\varepsilon(X^\varepsilon_s)]ds\right|^2_{\mathbb{R}^n}}
\\
& \qquad \qquad \qquad \displaystyle{+C\eta\mathbf{E}\left|\int_0^t \Sigma_2(\widehat{X}_s^{\varepsilon,\eta}, \widehat{Y}_s^{\varepsilon,\eta})dW_s^2\right|^2_{\mathbb{R}^n}+C\Delta}
\\
\leq & \displaystyle{C\left[\dfrac{t}{\Delta}\right]\sum\limits_{k=0}^{[t/\Delta]}\mathbf{E}\left|\int_{k\Delta}^{(k+1)\Delta}
[B_2(X^{\varepsilon,\eta}_{k\Delta}, \widehat{Y}_s^{\varepsilon,\eta})-\overline{B_2(X^{\varepsilon,\eta}_{k\Delta},Y)}^\varepsilon(X^{\varepsilon,\eta}_{k\Delta})]ds\right|_{\mathbb{R}^n}^2}
\\
& \displaystyle{  + Ct\Delta^2+Ct \max\limits_{0\leq s \leq t}\mathbf{E}|X_s^{\varepsilon,\eta}-\widehat{X}_s^{\varepsilon,\eta}|_{\mathbb{R}^n}^2
+Ct\int_0^t \mathbf{E}|\widehat{X}_s^{\varepsilon,\eta}-X_s^\varepsilon|_{\mathbb{R}^n}^2ds+C\eta+C\Delta}
\\
\leq & \displaystyle{\dfrac{C\eta}{\varepsilon\Delta}(1+\varepsilon+\varepsilon\max\limits_{0\leq s\leq t}\mathbf{E}|X_s^{\varepsilon,\eta}|_{\mathbb{R}^n}^2
+\varepsilon\max\limits_{0\leq s\leq t}\mathbf{E}|Y_s^{\varepsilon,\eta}|_{\mathbb{R}^n}^2)}
\\
& \displaystyle{ + Ct\Delta^2+Ct \max\limits_{0\leq s \leq t}\mathbf{E}|X_s^{\varepsilon,\eta}-\widehat{X}_s^{\varepsilon,\eta}|_{\mathbb{R}^n}^2
+Ct \int_0^t \mathbf{E}|\widehat{X}_s^{\varepsilon,\eta}-X_s^\varepsilon|_{\mathbb{R}^n}^2ds+C\eta+C\Delta}
\\
\leq & \displaystyle{C\dfrac{\eta}{\Delta}\left(\dfrac{1}{\varepsilon}+1+|x_0|_{\mathbb{R}^n}^2+|y_0|_{\mathbb{R}^n}^2\right) + Ct\Delta^2+Ct(T^2+T^2\varepsilon^2+1)\eta^{1-\kappa}
+C\Delta}
\\
& \qquad \qquad \qquad \displaystyle{+Ct\int_0^t \mathbf{E}|\widehat{X}_s^{\varepsilon,\eta}-X_s^\varepsilon|_{\mathbb{R}^n}^2ds \ .}
\end{array}$$}

Taking into account the choice of $\Delta$ in \eqref{Lemma:L2ClosessYandYhat:Eq:ChoiceOfDelta}, we further infer that when $\varepsilon, \eta>0$ are very small
$$
\mathbf{E}|\widehat{X}^{\varepsilon,\eta}_t-X^\varepsilon_t|_{\mathbb{R}^n}^2
\leq C\left(\dfrac{1}{\varepsilon\sqrt[4]{\ln(\eta^{-1})}}+\eta\sqrt[4]{\ln(\eta^{-1})}\right) + Ct\int_0^t \mathbf{E}|\widehat{X}_s^{\varepsilon,\eta}-X_s^\varepsilon|_{\mathbb{R}^n}^2ds \ .
$$

This ensures that
$$\mathbf{E}|\widehat{X}^{\varepsilon,\eta}_t-X^\varepsilon_t|_{\mathbb{R}^n}^2\leq C\left(\dfrac{1}{\varepsilon\sqrt[4]{\ln(\eta^{-1})}}+\eta\sqrt[4]{\ln(\eta^{-1})}\right)
\exp\left(Ct\right) \ ,$$
and thus, combining this with \eqref{Lemma:ClosenessXToXhat:Eq:ClosenessXToXhat}, we infer that as $\varepsilon,\eta>0$ are small we have,
for $0\leq t\leq T$,
$$\mathbf{E}|X^{\varepsilon,\eta}_t-X^\varepsilon_t|_{\mathbb{R}^n}^2\leq C\left(\dfrac{1}{\varepsilon\sqrt[4]{\ln(\eta^{-1})}}+\eta^{1-a}\right)
\exp\left(CT\right)+C(T^2+T^2\varepsilon^2+1)\eta^{1-\kappa} \ ,$$
so that we conclude with \eqref{Proposition:ConvergenceXEpsEtaToXEps:Eq:ClosenessXEpsEtaToXEps} by taking into account that
for any $a>0$ we have $\lim\limits_{\eta\rightarrow 0}\dfrac{\sqrt[4]{\ln(\eta^{-1})}}{\eta^{-a}}=0$.
\end{proof}

\begin{remark}
By using the corrector method in the next section, it is possible to remove $\kappa$
in the estimate \eqref{Proposition:ConvergenceXEpsEtaToXEps:Eq:ClosenessXEpsEtaToXEps}, and obtain a little bit
better upper bound $C\left(\dfrac{\eta^2}{\varepsilon^2}+\eta\right)$. However, the Khasminskii's method we use here
is more intuitive. See Lemma \ref{Lemma:ClosenessXToXhatCorrectorMethod} for a precise statement and proof. Also see Remark 2 in Section 5 for further discussion.
\end{remark}

\section{Normal deviations.}\label{sec:normal}

In this section we consider normal deviations of the process $X^{\varepsilon,\eta}(t)$ from the averaged motion $X^{\varepsilon}(t)$.
The method we use here is the corrector method. Similar techniques can be found in \cite{[KV86]}. To apply this method, a--priori
estimates of an auxiliary Poisson equation is needed, and there are various previous works dedicated to obtaining these estimates.
In the paper \cite{[HairerPavliotis2004]}, the authors considered the case when diffusion matrix is a scalar multiple of the identity matrix,
and the Poisson equation there corresponds to hypo--elliptic diffusions. Our analysis relies more on estimates obtained in \cite{[Pardoux-Veretennikov2]}
(also see \cite{[Pardoux-Veretennikov1]}, \cite{[Pardoux-Veretennikov3]}).

As in \eqref{Eq:RescaledDeviation}, we define
$$Z^{\varepsilon,\eta}(t)=\dfrac{X^{\varepsilon,\eta}(t)-X^{\varepsilon}(t)}{\sqrt{\eta}} \ .$$
By \eqref{Eq:DiffusionLimitTimeChangedStandardForm} and \eqref{Eq:AveragedSlowMotionGaussianMeasureDeterministic} we see that the
process $Z^{\varepsilon,\eta}(t)$ satisfies the integral equation
\begin{equation}\label{Eq:RescaledDeviationIntegralEquation}
Z^{\varepsilon,\eta}(t)=\dfrac{1}{\sqrt{\eta}}\int_0^t [B_2(X^{\varepsilon,\eta}_s, Y^{\varepsilon,\eta}_s)
-\overline{B_2(X^{\varepsilon}_s,Y)}^{\varepsilon}(X^{\varepsilon}_s)]ds+\int_0^t \Sigma_2(X^{\varepsilon,\eta}_s, Y^{\varepsilon,\eta}_s)dW_s^2 \ .
\end{equation}

Set
\begin{equation}\label{Eq:UProcess}
U^{\varepsilon,\eta}(t)=\dfrac{1}{\sqrt{\eta}}\int_0^t [B_2(X^{\varepsilon,\eta}_s, Y^{\varepsilon,\eta}_s)
-\overline{B_2(X^{\varepsilon}_s,Y)}^{\varepsilon}(X^{\varepsilon}_s)]ds
\end{equation}
and
\begin{equation}\label{Eq:VProcess}
V^{\varepsilon,\eta}(t)=\int_0^t \Sigma_2(X^{\varepsilon,\eta}_s, Y^{\varepsilon,\eta}_s)dW_s^2 \ .
\end{equation}

Let us introduce the infinitesimal generator of the OU process $\mathfrak{y}^{X,\varepsilon}(t)$
in \eqref{Eq:FastMotionDeterministicSystemRandomPerturbation} as the operator
\begin{equation}\label{Eq:FastProcessGenerator}
\mathcal{L}^{X,\varepsilon} f(Y)=\dfrac{\varepsilon}{2}\nabla_Y\cdot (A_1(X)\nabla_Y f(Y))+(B_1(X)-Y)\cdot \nabla_Y f(Y) \ ,
\end{equation}
and consider the auxiliary Poisson equations
\begin{equation}\label{Eq:AuxiliaryPDE}
\mathcal{L}^{X,\varepsilon} u_k(X,Y)=[B_2]_k(X,Y)-[\overline{B_2(X,Y)}^\varepsilon]_k(X)
\end{equation}
for $k=1,2,...,n$.

Since the variables $(X,Y)\in \mathbb{R}^n\times \mathbb{R}^m$, we are going to single out a unique solution
$u_k(X,Y)$ to \eqref{Eq:AuxiliaryPDE} by putting a restriction that for each $k=1,2,...,n$
\begin{equation}\label{Eq:AuxiliaryPDEUniquenessRestriction}
\int_{\mathbb{R}^n} u_k(X,Y)\mu^{X,\varepsilon}(dY)=0 \ .
\end{equation}

We have the following a--priori bounds for the solution $u_k(X,Y)$ based on \cite{[Pardoux-Veretennikov2]}.

\begin{lemma}\label{Lemma:PardouxVeretennnikovAPriorBoundsAuxiliaryPDE}
The solution $u_k(X,Y)\in \mathbf{C}^{(2)}(\mathbb{R}^n\times\mathbb{R}^m)$. Moreover,
there exist some integer $p>0$ and constant $C_k>0$ such that for each $k=1,2,...,n$ we have
\begin{equation}\label{Lemma:PardouxVeretennnikovAPriorBoundsAuxiliaryPDE:Eq:Bounds}
\begin{split}
 &\quad
|u_k(X,Y)|+|\nabla_X u_k(X,Y)|_{\mathbb{R}^n}+|\nabla^2_{X}u_k(X,Y)|_{\mathbb{R}^n\otimes \mathbb{R}^n}+|\nabla_Y u(X,Y)|_{\mathbb{R}^m}
  \\&\leq
C_k (1+|Y|_{\mathbb{R}^m}^p) \ .
\end{split}
\end{equation}
\end{lemma}

\begin{proof}
For simplicity of notations we will suppress the index $k$ throughout the proof.
Set $f(X,Y)=-\left([B_2]_k(X,Y)-[\overline{B_2(X,Y)}^\varepsilon]_k(X)\right)$.
The Poisson equation
\eqref{Eq:AuxiliaryPDE} can be written as
$$\mathcal{L}^{X,\varepsilon} u(X,Y)=-f(X,Y) \ .$$
An explicit representation of the solution $u(X,Y)$ can be formulated as
\begin{equation}\label{Lemma:PardouxVeretennnikovAPriorBoundsAuxiliaryPDE:Eq:RepSolPossion:u}
u(X,Y)=\int_0^\infty dt \int_{\mathbb{R}^m} dY' f(X, Y')p_t(Y, Y'; X) \ .
\end{equation}
Here $p_t(Y,Y';X)$ is the (parabolic) fundamental solution (transition probability density)
corresponding to the operator $\mathcal{L}^{X,\varepsilon}$, i.e.,
$$\dfrac{\partial}{\partial t}p_t(Y,Y'; X)=\mathcal{L}^{X,\varepsilon} p_t(Y, Y'; X) \ , \ p_0(Y,Y';X)=\delta(Y'-Y) \ .$$
The fact that $u_k(X,Y)\in \mathbf{C}^{(2)}(\mathbb{R}^n\times\mathbb{R}^m)$ is a consequence of Theorem 3 in \cite{[Pardoux-Veretennikov2]}.
Let us define
\begin{equation}\label{Lemma:PardouxVeretennnikovAPriorBoundsAuxiliaryPDE:Eq:TransitionDensityf}
\begin{array}{ll}
p_t(Y,f;X)&= \displaystyle{\int_{\mathbb{R}^m} f(X,Y')p_t(Y,Y';X)dY'}
\\
&= \displaystyle{\int_{\mathbb{R}^m} f(X,Y')[p_t(Y,Y';X)-p_{\infty}(Y';X)]dY' \ .}
\end{array}
\end{equation}

Here $p_{\infty}(Y; X)$ is the density function for the invariant measure $\mu^{X,\varepsilon}(dY)$ in
\eqref{Eq:InvariantMeasureOUSmallDiffusion}. Notice that the way we define the averaging
operator in \eqref{Eq:BarEpsOperator} guarantees that
\begin{equation}\label{Lemma:PardouxVeretennnikovAPriorBoundsAuxiliaryPDE:Eq:CenteringConditionf}
\displaystyle{\int_{\mathbb{R}^m} f(X,Y')p_{\infty}(Y'; X)dY'=0} \ ,
\end{equation}
which thus leads to the validity of the second equality in \eqref{Lemma:PardouxVeretennnikovAPriorBoundsAuxiliaryPDE:Eq:TransitionDensityf}.

From \eqref{Lemma:PardouxVeretennnikovAPriorBoundsAuxiliaryPDE:Eq:RepSolPossion:u}, and combining
\eqref{Lemma:PardouxVeretennnikovAPriorBoundsAuxiliaryPDE:Eq:TransitionDensityf}, we can write
\begin{equation}\label{Lemma:PardouxVeretennnikovAPriorBoundsAuxiliaryPDE:Eq:RepSolPossion:uAlternate}
u(X,Y)=\int_0^1 dt p_t(Y,f;X)+\int_1^\infty dt p_t(Y,f;X) \ .
\end{equation}

Let us define $p_t^{(j)}(Y,Y';X)=\nabla_X^j p_t(Y,Y'; X)$ to be the $j$--derivative of $p_t(Y,Y';\break X)$ with respect to $X$, and it is
a tensor with $j$ indices. In a similar fashion, from \eqref{Lemma:PardouxVeretennnikovAPriorBoundsAuxiliaryPDE:Eq:TransitionDensityf},
we define $p_t^{(j)}(Y, f; X)=\nabla_X^j \left(\displaystyle{\int_{\mathbb{R}^m} f(X, Y')p_t(Y, Y'; X)dY'}\right)$.
Thus $p_t^{(1)}(Y, f; X)$ is an $n$--dimensional vector, and $p_t^{(2)}(Y, f; X)$ is an $n\times n$ matrix.
Thus we have
\begin{equation}\label{Lemma:PardouxVeretennnikovAPriorBoundsAuxiliaryPDE:Eq:RepSolPossion:uX}
\nabla_X u(X,Y)=\int_0^1 dt p_t^{(1)}(Y,f;X)+\int_1^\infty dt p_t^{(1)}(Y,f;X) \ ,
\end{equation}
\begin{equation}\label{Lemma:PardouxVeretennnikovAPriorBoundsAuxiliaryPDE:Eq:RepSolPossion:uXX}
\nabla_X^2 u(X,Y)=\int_0^1 dt p_t^{(2)}(Y,f;X)+\int_1^\infty dt p_t^{(2)}(Y,f;X) \ ,
\end{equation}
\begin{equation}\label{Lemma:PardouxVeretennnikovAPriorBoundsAuxiliaryPDE:Eq:RepSolPossion:uY}
\nabla_Y u(X,Y)=\int_0^1 dt \nabla_Y p_t(Y,f;X)+\int_1^\infty dt \nabla_Y p_t(Y,f;X) \ .
\end{equation}

By the estimates (14) and (15) from Theorem 2 in \cite{[Pardoux-Veretennikov2]}, taking into account the
centering condition \eqref{Lemma:PardouxVeretennnikovAPriorBoundsAuxiliaryPDE:Eq:CenteringConditionf}, we see that
for any $k>0$ there exist $C,p>0$ such that for all $t\geq 1$ we have
\begin{equation}\label{Lemma:PardouxVeretennnikovAPriorBoundsAuxiliaryPDE:Eq:SecondIntegral1}
|p_t^{(1)}(Y,f;X)|_{\mathbb{R}^n}\leq C\dfrac{1+|Y|_{\mathbb{R}^m}^p}{(1+t)^k} \ ;
\end{equation}
\begin{equation}\label{Lemma:PardouxVeretennnikovAPriorBoundsAuxiliaryPDE:Eq:SecondIntegral2}
|p_t^{(2)}(Y,f;X)|_{\mathbb{R}^n\otimes \mathbb{R}^n}\leq C\dfrac{1+|Y|_{\mathbb{R}^m}^p}{(1+t)^k} \ ;
\end{equation}
and
\begin{equation}\label{Lemma:PardouxVeretennnikovAPriorBoundsAuxiliaryPDE:Eq:SecondIntegral3}
|\nabla_Y p_t(Y,f;X)|_{\mathbb{R}^m}\leq C\dfrac{1+|Y|_{\mathbb{R}^m}^p}{(1+t)^k} \ .
\end{equation}

Now we consider standard estimates, including derivatives with respect to $X$, for the integral $\displaystyle{\int_t^1 p_s(Y,f;X)ds}$
which solves a Cauchy problem for a parabolic equation in the region $[0,1]\times \mathbb{R}^m$ (with an initial value at $t=1$). In fact,
let
$$v(t,Y;X):=\displaystyle{\int_t^1 p_s(Y,f;X)ds}=\displaystyle{\int_t^1 ds\int_{\mathbb{R}^m} f(X,Y')p_s(Y,Y';X)dY'} \ ,$$
then by using Duhamel's principle we see that $v(t,Y;X)=\widetilde{v}(1-t, Y; X)$, where
$\widetilde{v}(s,Y; X)$ is the solution of the Cauchy problem
$$\dfrac{\partial \widetilde{v}}{\partial s}-\mathcal{L}^X \widetilde{v}=f(X,Y) \ , \ \widetilde{v}(0,Y;X)=0 \  .$$
We can apply standard parabolic estimates to the solution $\widetilde{v}$ so that by taking into account the compactly supportedness of
$f(X,Y)$ in terms of $X$ and $Y$, we get
\begin{equation}\label{Lemma:PardouxVeretennnikovAPriorBoundsAuxiliaryPDE:Eq:FirstIntegral}
|\nabla_X \widetilde{v}(1,Y;X)|_{\mathbb{R}^n}+|\nabla^2_X \widetilde{v}(1,Y;X)|_{\mathbb{R}^n\otimes \mathbb{R}^n}+
|\nabla_Y \widetilde{v}(1,Y;X)|_{\mathbb{R}^m}\leq C<\infty \ .
\end{equation}

By applying \eqref{Lemma:PardouxVeretennnikovAPriorBoundsAuxiliaryPDE:Eq:SecondIntegral1}, \eqref{Lemma:PardouxVeretennnikovAPriorBoundsAuxiliaryPDE:Eq:SecondIntegral2},
\eqref{Lemma:PardouxVeretennnikovAPriorBoundsAuxiliaryPDE:Eq:SecondIntegral3} to the second integrals in
\eqref{Lemma:PardouxVeretennnikovAPriorBoundsAuxiliaryPDE:Eq:RepSolPossion:uX},
\eqref{Lemma:PardouxVeretennnikovAPriorBoundsAuxiliaryPDE:Eq:RepSolPossion:uXX},
\eqref{Lemma:PardouxVeretennnikovAPriorBoundsAuxiliaryPDE:Eq:RepSolPossion:uY}, and
\eqref{Lemma:PardouxVeretennnikovAPriorBoundsAuxiliaryPDE:Eq:FirstIntegral} to the first integrals in
\eqref{Lemma:PardouxVeretennnikovAPriorBoundsAuxiliaryPDE:Eq:RepSolPossion:uX},
\eqref{Lemma:PardouxVeretennnikovAPriorBoundsAuxiliaryPDE:Eq:RepSolPossion:uXX},
\eqref{Lemma:PardouxVeretennnikovAPriorBoundsAuxiliaryPDE:Eq:RepSolPossion:uY},
we conclude
\eqref{Lemma:PardouxVeretennnikovAPriorBoundsAuxiliaryPDE:Eq:Bounds}.
\end{proof}

The next Lemma is a reproduction of Proposition 4.2 estimate (4.5) in \cite{[CerraiKhasminskii2009]}.

\begin{lemma}\label{Lemma:MomentBoundOUCerraiKhasminskii}
For any $p\geq 1$ and any $T>0$, there exist constant $C=C(p,T)>0$, such that for all
$x_0\in \mathbb{R}^n , y_0\in \mathbb{R}^m$ we have
\begin{equation}\label{Lemma:MomentBoundOUCerraiKhasminskii:Eq:BoundsY}
\int_0^T \mathbf{E}|Y^{\varepsilon,\eta}_t|_{\mathbb{R}^m}^p dt\leq C(1+|y_0|_{\mathbb{R}^m}^p) \ .
\end{equation}
\end{lemma}

\begin{proof}
First of all, we can write the equation
\eqref{Eq:DiffusionLimitTimeChangedStandardForm} for $X^{\varepsilon,\eta}_t$ in an integral form as
$$
X^{\varepsilon,\eta}_t=x_0+\int_0^t B_2(X^{\varepsilon,\eta}_s, Y^{\varepsilon,\eta}_s)ds+\sqrt{\eta}\int_0^t
\Sigma_2(X^{\varepsilon,\eta}_s, Y^{\varepsilon,\eta}_s)dW_s^2 \ .
$$

Therefore we know that
$$\begin{array}{ll}
\mathbf{E}|X^{\varepsilon,\eta}_t|_{\mathbb{R}^n}^p
\leq &\displaystyle{
C(p,t)\left(
  |x_0|_{\mathbb{R}^n}^p
  +
  \mathbf{E}\left|\int_0^t B_2(X^{\varepsilon,\eta}_s, Y^{\varepsilon,\eta}_s)ds\right|_{\mathbb{R}^n}^p
\right.
}
\\
&\quad
\displaystyle{ \left.
+\eta^{p/2}\mathbf{E}\left|\int_0^t
\Sigma_2(X^{\varepsilon,\eta}_s, Y^{\varepsilon,\eta}_s)dW_s^2\right|_{\mathbb{R}^n}^p
\right)
\ .}
\\
\end{array}$$

By making use of the identity \eqref{Eq:KiferMatrixStochasticIntegralItoIsometry}, as well as
the Burkholder--Davis--Gundy inequality (see \cite[Corollary IV.4.2]{[RY99]}), we have
\begin{equation}\label{Eq:BDGInequalityControlStochasticIntegralTerm}
\mathbf{E}\left|\int_0^t \Sigma_2(X_s^{\varepsilon,\eta}, Y_s^{\varepsilon,\eta})dW_s^2\right|_{\mathbb{R}^n}^p\leq
C(p,t)\left(\int_0^t \mathbf{E}\|\Sigma_2(X_s^{\varepsilon,\eta}, Y_s^{\varepsilon,\eta})\|_{\mathbb{R}^n\otimes \mathbb{R}^n}^2ds\right)^{p/2} \ .
\end{equation}

From the mild form \eqref{Eq:MildFormYprocess} of the process $Y_t^{\varepsilon,\eta}$ we have
\begin{equation*}
Y_t^{\varepsilon,\eta}=e^{-\frac{\varepsilon}{\eta}t}y_0+\dfrac{\varepsilon}{\eta}\int_0^t e^{-\frac{\varepsilon}{\eta}(t-s)} B_1(X_s^{\varepsilon,\eta})ds
+\dfrac{\varepsilon}{\sqrt{\eta}}\int_0^t e^{-\frac{\varepsilon}{\eta}(t-s)}\Sigma_1(X_s^{\varepsilon,\eta})dW_s^1 \ .
\end{equation*}

Set $\Gamma(t)=\displaystyle{\dfrac{\varepsilon}{\sqrt{\eta}}\int_0^t e^{-\frac{\varepsilon}{\eta}(t-s)}\Sigma_1(X_s^{\varepsilon,\eta})dW_s^1}$ and $\Lambda(t)=Y^{\varepsilon,\eta}_t-\Gamma(t)$.
Then we have
$$d\Lambda(t)=-\dfrac{\varepsilon}{\eta}[\Lambda(t)+B_1(X^{\varepsilon,\eta}_t)]dt \ , \ \Lambda(0)=y_0 \ ,$$
which gives
$$\begin{array}{ll}
\displaystyle{\dfrac{1}{p}\dfrac{d}{dt}|\Lambda(t)|_{\mathbb{R}^n}^p } & =\displaystyle{\left\langle \Lambda(t), -\dfrac{\varepsilon}{\eta}[\Lambda(t)+B_1(X^{\varepsilon,\eta}_t)]\right\rangle_{\mathbb{R}^n}|\Lambda(t)|_{\mathbb{R}^n}^{p-2} }
\\
& \displaystyle{\leq -\dfrac{\varepsilon}{\eta}|\Lambda(t)|_{\mathbb{R}^n}^p-\dfrac{\varepsilon}{\eta}\langle |\Lambda(t)|_{\mathbb{R}^n}^{p-2} \Lambda(t), B_1(X^{\varepsilon,\eta}_t)\rangle_{\mathbb{R}^n}}
\\
& \displaystyle{\leq -\dfrac{\varepsilon}{\eta}|\Lambda(t)|_{\mathbb{R}^n}^p+\dfrac{\varepsilon}{\eta}\left(\dfrac{p-1}{p}|\Lambda(t)|_{\mathbb{R}^n}^p+\dfrac{1}{p}|B_1(X^{\varepsilon,\eta}_t)|_{\mathbb{R}^n}^p\right)}
\\
& \displaystyle{= -\dfrac{\varepsilon}{p\eta}|\Lambda(t)|_{\mathbb{R}^n}^p+\dfrac{C_{p,t}\varepsilon}{\eta} \ .}
\end{array}$$

Therefore by Gronwall inequality we know that for $0\leq t\leq T$ we have
$$|\Lambda(t)|_{\mathbb{R}^n}^p\leq Ce^{-\frac{\varepsilon}{\eta}t}|y_0|_{\mathbb{R}^m}^p+2CT\leq C(1+|y_0|_{\mathbb{R}^m}^p) \ .$$

It remains to estimate $\mathbf{E}|\Gamma(t)|_{\mathbb{R}^m}^p$. Again, by the identity \eqref{Eq:KiferMatrixStochasticIntegralItoIsometry}
as well as the Burkholder--Davis--Gundy inequality, we have
$$
\mathbf{E}|\Gamma(t)|_{\mathbb{R}^m}^p = \dfrac{\varepsilon^p}{\eta^{p/2}} e^{-\frac{p\varepsilon}{\eta}t}
\displaystyle{\left(\int_0^t e^{\frac{2\varepsilon}{\eta}s}\mathbf{E}\|\Sigma_1(X^{\varepsilon,\eta}_s)\|_{\mathbb{R}^n\otimes \mathbb{R}^n}^2ds\right)^{p/2}}
\leq C \varepsilon^{p/2} \ .$$

Thus we obtain
\begin{equation}\label{Lemma:MomentBoundOUCerraiKhasminskii:Eq:FixTimeBoundsY}
\mathbf{E}|Y^{\varepsilon,\eta}_t|_{\mathbb{R}^m}^p \leq C(\mathbf{E}|\Lambda(t)|_{\mathbb{R}^m}^p+\mathbf{E}|\Gamma(t)|_{\mathbb{R}^m}^p)\leq C(1+|y_0|_{\mathbb{R}^m}^p) \ ,
\end{equation}
which leads to \eqref{Lemma:MomentBoundOUCerraiKhasminskii:Eq:BoundsY} by integrating on $[0,T]$.
\end{proof}

The next Lemma is about how averaging principle is used to evaluate integrals.

\begin{lemma}\label{Lemma:AveragingPrincipleIntegralLimit}
Let $K(X,Y)$ be a Lipschitz continuous function in $X, Y$ such that
$|K(X_1,Y_1)-K(X_2,Y_2)|\leq C(|X_1-X_2|_{\mathbb{R}^n}+|Y_1-Y_2|_{\mathbb{R}^m})$ and
$|K(X,Y)|\leq C(1+|Y|_{\mathbb{R}^m}^p)$ for some $C, p>0$.
For each $\varepsilon>0$ fixed, as $\eta \rightarrow 0$ we have
\begin{equation}\label{Lemma:AveragingPrincipleIntegralLimit:Eq:AveragingPrinciple}
\mathbf{E}\left|\int_0^t K(X^{\varepsilon,\eta}_s, Y^{\varepsilon,\eta}_s)ds-\int_0^t ds\int_{\mathbb{R}^m}K(X^{\varepsilon}_s,Y)\mu^{\varepsilon,X^{\varepsilon}_s}(dY)\right|\rightarrow 0 \ .
\end{equation}
\end{lemma}

\begin{proof}
By the estimates \eqref{Lemma:L2ClosessYandYhat:Eq:L2ClosessYandYhat} and \eqref{Proposition:ConvergenceXEpsEtaToXEps:Eq:ClosenessXEpsEtaToXEps}
 and making use of the fact that
$K(X,Y)$ is Lipschitz continuous in $X$ and $Y$, we have
$$
\begin{array}{ll}
&\displaystyle{\mathbf{E}\left|\int_0^t K(X^{\varepsilon,\eta}_s, Y^{\varepsilon,\eta}_s)ds-
\int_0^t K(X^{\varepsilon}_s, \widehat{Y}^{\varepsilon,\eta}_s)ds\right|}
\\
\leq & \displaystyle{C\mathbf{E}\int_0^t \left(|X^{\varepsilon,\eta}_s-X^{\varepsilon}_s|_{\mathbb{R}^n}+|Y^{\varepsilon,\eta}_s-\widehat{Y}^{\varepsilon,\eta}_s|_{\mathbb{R}^m}\right)ds}
\\
\leq & \displaystyle{C\int_0^t \left[\left(\mathbf{E}|X^{\varepsilon,\eta}_s-X^{\varepsilon}_s|^2_{\mathbb{R}^n}\right)^{1/2}+\left(\mathbf{E}|Y^{\varepsilon,\eta}_s-\widehat{Y}^{\varepsilon,\eta}_s|^2_{\mathbb{R}^m}\right)^{1/2}\right]ds \ ,}
\end{array}$$
so that
\begin{equation}\label{Lemma:AveragingPrincipleIntegralLimit:Eq:Estimate1}
\mathbf{E}\left|\int_0^t K(X^{\varepsilon,\eta}_s, Y^{\varepsilon,\eta}_s)ds-
\int_0^t K(X^{\varepsilon}_s, \widehat{Y}^{\varepsilon,\eta}_s)ds\right|\rightarrow 0
\end{equation}
as $\eta \rightarrow 0$. Here the process $\widehat{Y}^{\varepsilon,\eta}_t$ is defined as in \eqref{Eq:DiffusionLimitTimeChanged:AuxiliaryProcessYhat}.
Reasoning as in \eqref{Proposition:ConvergenceXEpsEtaToXEps:Eq:ConvergenceMixingPartFastMotion}, we have, that for each $k=0,1,...$
and each interval $[k\Delta, (k+1)\Delta]$, $\Delta>0$,
$$
\begin{array}{ll}
& \displaystyle{\mathbf{E}\left|
\int_{k\Delta}^{(k+1)\Delta} K(X^{\varepsilon}_s, \widehat{Y}^{\varepsilon,\eta}_s)ds
-
\int_{k\Delta}^{(k+1)\Delta} ds \int_{\mathbb{R}^m}K(X^{\varepsilon}_s,Y)\mu^{\varepsilon,X^{\varepsilon,\eta}_{k\Delta}}(dY)
\right|^2}
\\
\leq & \displaystyle{C\Delta\eta\left(\dfrac{1}{\varepsilon}+1+|x_0|_{\mathbb{R}^n}^2+|y_0|_{\mathbb{R}^m}^2\right)} \ ,
\end{array}
$$
so that if we divide the interval $[0,t]$ into intervals each of size $\Delta=\eta\sqrt[4]{\ln(\eta^{-1})}$,
with the same argument as in Proposition \ref{Proposition:ConvergenceXEpsEtaToXEps} we derive
\begin{equation}\label{Lemma:AveragingPrincipleIntegralLimit:Eq:Estimate2}
\mathbf{E}\left|
\int_0^t K(X^{\varepsilon}_s, \widehat{Y}^{\varepsilon,\eta}_s)ds
-
\int_0^t ds \int_{\mathbb{R}^m}K(X^{\varepsilon}_s,Y)\mu^{\varepsilon,X^{\varepsilon,\eta}_{k\Delta}}(dY)
\right|
\rightarrow 0
\end{equation}
as $\eta \rightarrow 0$. Reasoning as in Lemma \ref{Lemma:RegularityBarEpsilonOperator}, and making use of the fact that
$|K(X,Y)|\leq C(1+|Y|_{\mathbb{R}^m}^p)$ for some $C, p>0$, we have
$$
\begin{array}{ll}
\hspace*{-5pt}& \displaystyle{\mathbf{E}\left|
\int_{k\Delta}^{(k+1)\Delta} ds \int_{\mathbb{R}^m}K(X^{\varepsilon}_s,Y)\mu^{\varepsilon,X^{\varepsilon,\eta}_{k\Delta}}(dY)
-
\int_{k\Delta}^{(k+1)\Delta} ds \int_{\mathbb{R}^m}K(X^{\varepsilon}_s,Y)\mu^{\varepsilon,X^{\varepsilon}_{s}}(dY)
\right|}
\\
\hspace*{-5pt}& \le \displaystyle{\int_{k\Delta}^{(k+1)\Delta} ds
\mathbf{E} \left|\int_{\mathbb{R}^m}K(X^{\varepsilon}_s,Y)\mu^{\varepsilon,X^{\varepsilon,\eta}_{k\Delta}}(dY)-\int_{\mathbb{R}^m}K(X^{\varepsilon}_s,Y)\mu^{\varepsilon,X^{\varepsilon}_{s}}(dY)\right|}
\\
\hspace*{-5pt}& \le C \Delta \displaystyle{\left(\max\limits_{k\Delta\leq s\leq (k+1)\Delta}\mathbf{E}|X^{\varepsilon,\eta}_{k\Delta}-X^{\varepsilon}_s|_{\mathbb{R}^n}\right)}
\\
\hspace*{-5pt}& \le C \Delta \displaystyle{\max\limits_{k\Delta\leq s\leq (k+1)\Delta}\left(\mathbf{E}|X^{\varepsilon,\eta}_{k\Delta}-X^{\varepsilon, \eta}_s|^2_{\mathbb{R}^n}
+\mathbf{E}|X^{\varepsilon, \eta}_s-X^\varepsilon_s|^2_{\mathbb{R}^n}\right)^{1/2}} \ .
\end{array}
$$
From here, using \eqref{Lemma:L2ClosessYandYhat:Eq:ClosenessDiscretizedSlowMotion} and
\eqref{Proposition:ConvergenceXEpsEtaToXEps:Eq:ClosenessXEpsEtaToXEps}, and summing over all intervals of the form $[k\Delta, (k+1)\Delta]$
for $k=0,1,...,N-1$
we know that
\begin{equation}\label{Lemma:AveragingPrincipleIntegralLimit:Eq:Estimate3}
\mathbf{E}\left|
\int_{0}^{t} ds \int_{\mathbb{R}^m}K(X^{\varepsilon}_s,Y)\mu^{\varepsilon,X^{\varepsilon,\eta}_{k\Delta}}(dY)
-
\int_{0}^{t} ds \int_{\mathbb{R}^m}K(X^{\varepsilon}_s,Y)\mu^{\varepsilon,X^{\varepsilon}_{s}}(dY)
\right|\rightarrow 0
\end{equation}
as $\eta \rightarrow 0$. Finally \eqref{Lemma:AveragingPrincipleIntegralLimit:Eq:Estimate1}, \eqref{Lemma:AveragingPrincipleIntegralLimit:Eq:Estimate2}
and \eqref{Lemma:AveragingPrincipleIntegralLimit:Eq:Estimate3} conclude \eqref{Lemma:AveragingPrincipleIntegralLimit:Eq:AveragingPrinciple}.
\end{proof}

Set
\begin{equation}\label{Eq:U1Process}
U_1^{\varepsilon,\eta}(t)=\dfrac{1}{\sqrt{\eta}}\int_0^t [B_2(X^{\varepsilon,\eta}_s, Y^{\varepsilon,\eta}_s)
-\overline{B_2(X^{\varepsilon,\eta}_s,Y)}^{\varepsilon}(X^{\varepsilon,\eta}_s)]ds \ ,
\end{equation}
\begin{equation}\label{Eq:U2Process}
U_2^{\varepsilon,\eta}(t)=\dfrac{1}{\sqrt{\eta}}\int_0^t [\overline{B_2(X^{\varepsilon,\eta}_s,Y)}^{\varepsilon}(X^{\varepsilon,\eta}_s)
-\overline{B_2(X^{\varepsilon}_s,Y)}^{\varepsilon}(X^{\varepsilon}_s)]ds \ .
\end{equation}

The next Lemma characterizes the weak convergence of $U_1^{\varepsilon,\eta}(t)$ as $\eta \rightarrow 0$ to a mean zero Gaussian process.

\begin{lemma}\label{Lemma:WeakConvergenceU1ToGaussian}
For each $\varepsilon>0$, as $\eta \rightarrow 0$, for $0\leq t \leq T$, the family of processes $U_1^{\varepsilon,\eta}(t)$ converges weakly
to a Gaussian process $N_1^\varepsilon(t)$ with mean $0$ and covariance matrix $A^\varepsilon(t)=(a_{i,j}^{\varepsilon}(t))_{1\leq i,j\leq n}$, so that
\begin{equation}\label{Lemma:WeakConvergenceU1ToGaussian:Eq:CovarianceMatrix}
a_{i,j}^{\varepsilon}(t)=\displaystyle{\int_0^t  \overline{\nabla_Y^T u_i(X^{\varepsilon}_s, Y)}^\varepsilon(X^\varepsilon_s)
\Sigma_1(X^{\varepsilon}_s)\Sigma_1^T(X^{\varepsilon}_s)
 \overline{\nabla_Y u_j(X^{\varepsilon}_s, Y)}^\varepsilon(X^\varepsilon_s)ds} \ .
\end{equation}
\end{lemma}

\begin{proof}
Let $u_k(X,Y)$ be the solution to \eqref{Eq:AuxiliaryPDE}, $k=1,2,...,n$.
Let us then apply It\^{o}'s formula to $u_k(X^{\varepsilon,\eta}_t, Y^{\varepsilon,\eta}_t)$, and we get
$$\begin{array}{ll}
&u_k(X^{\varepsilon,\eta}_t,Y^{\varepsilon,\eta}_t)-u_k(X^{\varepsilon,\eta}_0,Y^{\varepsilon,\eta}_0) \\
=& \displaystyle{\int_0^t \nabla_X u_k(X^{\varepsilon,\eta}_s, Y^{\varepsilon,\eta}_s)\cdot B_1(X^{\varepsilon,\eta}_s, Y^{\varepsilon,\eta}_s)ds}
\end{array}$$$$\begin{array}{ll}& \qquad \displaystyle{+\dfrac{\eta}{2} \int_0^t \nabla_X\cdot(A_2(X^{\varepsilon,\eta}_s,Y^{\varepsilon,\eta}_s)\nabla_X u_k(X^{\varepsilon,\eta}_s, Y^{\varepsilon,\eta}_s))ds}
\\
& \qquad \displaystyle{+\sqrt{\eta}\int_0^t \nabla_X u_k(X^{\varepsilon,\eta}_s, Y^{\varepsilon,\eta}_s)\cdot \Sigma_2(X^{\varepsilon,\eta}_s, Y^{\varepsilon,\eta}_s)dW_s^2}
\\
& \displaystyle{\qquad +\dfrac{\varepsilon}{\eta}\int_0^t \mathcal{L}^{X^{\varepsilon,\eta}_s}u_k(X^{\varepsilon,\eta}_s, Y^{\varepsilon,\eta}_s)ds}
\\
& \displaystyle{\qquad +\sqrt{\varepsilon}\left(\dfrac{\varepsilon}{\eta}\right)^{1/2}\int_0^t \nabla_Y u_k(X^{\varepsilon,\eta}_s, Y^{\varepsilon,\eta}_s)\cdot \Sigma_1(X^{\varepsilon,\eta}_s)dW_s^1 \ .}
\end{array}$$

From the above identity and the Poisson equation \eqref{Eq:AuxiliaryPDE} we obtain that
\begin{equation}\label{Lemma:WeakConvergenceU1ToGaussian:Eq:EstimatesOfCorrection}
\begin{array}{ll}
& \displaystyle{\int_0^t \left([B_2]_k(X^{\varepsilon,\eta}_s, Y^{\varepsilon,\eta}_s)-[\overline{B_2(X^{\varepsilon,\eta}_s, Y^{\varepsilon,\eta}_s)}^\varepsilon]_k(X^{\varepsilon,\eta}_s)\right)ds}
\\
= & \dfrac{\eta}{\varepsilon}[u_k(X^{\varepsilon,\eta}_t,Y^{\varepsilon,\eta}_t)-u_k(X^{\varepsilon,\eta}_0,Y^{\varepsilon,\eta}_0)]
\\
& \qquad -\displaystyle{\dfrac{\eta}{\varepsilon}\int_0^t \nabla_X u_k(X^{\varepsilon,\eta}_s, Y^{\varepsilon,\eta}_s)\cdot B_1(X^{\varepsilon,\eta}_s, Y^{\varepsilon,\eta}_s)ds}
\\
& \qquad \displaystyle{+\dfrac{\eta}{2}\cdot \dfrac{\eta}{\varepsilon} \int_0^t \nabla_X\cdot(A_2(X^{\varepsilon,\eta}_s,Y^{\varepsilon,\eta}_s)\nabla_X u_k(X^{\varepsilon,\eta}_s, Y^{\varepsilon,\eta}_s))ds}
\\
& \qquad \displaystyle{-\sqrt{\eta}\cdot \dfrac{\eta}{\varepsilon}\int_0^t \nabla_X u_k(X^{\varepsilon,\eta}_s, Y^{\varepsilon,\eta}_s)\cdot \Sigma_2(X^{\varepsilon,\eta}_s, Y^{\varepsilon,\eta}_s)dW_s^2}
\\
& \displaystyle{\qquad +\sqrt{\varepsilon}\left(\dfrac{\eta}{\varepsilon}\right)^{1/2}\int_0^t \nabla_Y u_k(X^{\varepsilon,\eta}_s, Y^{\varepsilon,\eta}_s)\cdot \Sigma_1(X^{\varepsilon,\eta}_s)dW_s^1 \ .}
\end{array}
\end{equation}

Therefore
\begin{equation}\label{Eq:CorrectorEquationSlowMotion}
\begin{array}{lll}
[U_1^{\varepsilon,\eta}]_k(t)&=& \displaystyle{\dfrac{1}{\sqrt{\eta}}
\int_0^t \left([B_2]_k(X^{\varepsilon,\eta}_s, Y^{\varepsilon,\eta}_s)-[\overline{B_2(X^{\varepsilon,\eta}_s, Y^{\varepsilon,\eta}_s)}^\varepsilon]_k(X^{\varepsilon,\eta}_s)\right)ds}
\\
\\
&= & \dfrac{\sqrt{\eta}}{\varepsilon}[u_k(X^{\varepsilon,\eta}_t,Y^{\varepsilon,\eta}_t)-u_k(X^{\varepsilon,\eta}_0,Y^{\varepsilon,\eta}_0)]
\\
&& \qquad -\displaystyle{\dfrac{\sqrt{\eta}}{\varepsilon}\int_0^t \nabla_X u_k(X^{\varepsilon,\eta}_s, Y^{\varepsilon,\eta}_s)\cdot B_1(X^{\varepsilon,\eta}_s, Y^{\varepsilon,\eta}_s)ds}
\\
&& \qquad \displaystyle{+\dfrac{\sqrt{\eta}}{2}\cdot \dfrac{\eta}{\varepsilon} \int_0^t \nabla_X\cdot(A_2(X^{\varepsilon,\eta}_s,Y^{\varepsilon,\eta}_s)\nabla_X u_k(X^{\varepsilon,\eta}_s, Y^{\varepsilon,\eta}_s))ds}
\\
&& \qquad \displaystyle{- \dfrac{\eta}{\varepsilon}\int_0^t \nabla_X u_k(X^{\varepsilon,\eta}_s, Y^{\varepsilon,\eta}_s)\cdot \Sigma_2(X^{\varepsilon,\eta}_s, Y^{\varepsilon,\eta}_s)dW_s^2}
\\
&& \displaystyle{\qquad +\int_0^t \nabla_Y u_k(X^{\varepsilon,\eta}_s, Y^{\varepsilon,\eta}_s)\cdot \Sigma_1(X^{\varepsilon,\eta}_s)dW_s^1}
\\
&=& (I)+(II)+(III) \ .
\end{array}
\end{equation}

Here
$$\begin{array}{ll}
(I)& =\displaystyle{\dfrac{\sqrt{\eta}}{\varepsilon}[u_k(X^{\varepsilon,\eta}_t,Y^{\varepsilon,\eta}_t)-u_k(X^{\varepsilon,\eta}_0,Y^{\varepsilon,\eta}_0)]
}
\\
&
\displaystyle{\qquad
-
\dfrac{\sqrt{\eta}}{\varepsilon}\int_0^t \nabla_X u_k(X^{\varepsilon,\eta}_s, Y^{\varepsilon,\eta}_s)\cdot B_1(X^{\varepsilon,\eta}_s, Y^{\varepsilon,\eta}_s)ds}
\\
& \displaystyle{\qquad
 +\dfrac{\sqrt{\eta}}{2}\cdot \dfrac{\eta}{\varepsilon} \int_0^t
\nabla_X\cdot(A_2(X^{\varepsilon,\eta}_s,Y^{\varepsilon,\eta}_s)\nabla_X u_k(X^{\varepsilon,\eta}_s, Y^{\varepsilon,\eta}_s))ds \ ,}
\\
(II)& =\displaystyle{- \dfrac{\eta}{\varepsilon}\int_0^t \nabla_X u_k(X^{\varepsilon,\eta}_s, Y^{\varepsilon,\eta}_s)\cdot \Sigma_2(X^{\varepsilon,\eta}_s, Y^{\varepsilon,\eta}_s)dW_s^2 \ ,}
\\
(III) & = \displaystyle{\int_0^t \nabla_Y u_k(X^{\varepsilon,\eta}_s, Y^{\varepsilon,\eta}_s)\cdot \Sigma_1(X^{\varepsilon,\eta}_s)dW_s^1 \ .}
\end{array}$$
We conclude from Lemma \ref{Lemma:PardouxVeretennnikovAPriorBoundsAuxiliaryPDE} that
$$\mathbf{E}|(I)|^2 \leq C\dfrac{\eta}{\varepsilon^2}\left(1+|x_0|^2+|y_0|^2+\left(\mathbf{E}\displaystyle{\int_0^t |Y_s^{\varepsilon,\eta}|_{\mathbb{R}^m}^pds}\right)^2\right) \ .$$
Moreover, by combining Lemma \ref{Lemma:PardouxVeretennnikovAPriorBoundsAuxiliaryPDE} as well as \eqref{Eq:KiferMatrixStochasticIntegralItoIsometry}
we also see that
$$\mathbf{E}|(II)|^2\leq C \dfrac{\eta^2}{\varepsilon^2}\left(1+\left(\displaystyle{\mathbf{E} \int_0^t |Y_s^{\varepsilon,\eta}|_{\mathbb{R}^m}^p ds}\right)^2\right) \ .$$

Making use of Lemma \ref{Lemma:MomentBoundOUCerraiKhasminskii} the estimate
\eqref{Lemma:MomentBoundOUCerraiKhasminskii:Eq:BoundsY}, we know that
$$\mathbf{E}(|(I)|^2+|(II)|^2)\rightarrow 0$$
as $\eta \rightarrow 0$.

Now we look at $(III)$. In fact, the term
$$(III) = \mathcal{M}_k^{\varepsilon,\eta}(t):=\displaystyle{\int_0^t \nabla_Y u_k(X^{\varepsilon,\eta}_s, Y^{\varepsilon,\eta}_s)\cdot \Sigma_1(X^{\varepsilon,\eta}_s)dW_s^1}$$
is a martingale with mean $0$ and quadratic variation
$$\begin{array}{ll}
\langle \mathcal{M}_k^{\varepsilon,\eta} , \mathcal{M}_k^{\varepsilon,\eta}\rangle_t
 & =
a_{k,k}^{\varepsilon,\eta}(t)
\\
 &
 \hspace{-.4in}
 =
\displaystyle{\int_0^t \nabla_Y^T u_k(X^{\varepsilon,\eta}_s, Y^{\varepsilon,\eta}_s)
\Sigma_1(X^{\varepsilon,\eta}_s)\Sigma_1^T(X^{\varepsilon,\eta}_s)
\nabla_Y u_k(X^{\varepsilon,\eta}_s, Y^{\varepsilon,\eta}_s)ds} \ .
\end{array}$$

Set
$$a_{k,k}^{\varepsilon}(t)=\displaystyle{\int_0^t  \overline{\nabla_Y^T u_k(X^{\varepsilon}_s, Y)}^\varepsilon(X^\varepsilon_s)
\Sigma_1(X^{\varepsilon}_s)\Sigma_1^T(X^{\varepsilon}_s)
 \overline{\nabla_Y u_k(X^{\varepsilon}_s, Y)}^\varepsilon(X^\varepsilon_s)ds} \ .$$

Making use of Lemma \ref{Lemma:AveragingPrincipleIntegralLimit}, we know that as $\eta\rightarrow 0$,
for any $R>0$ we have the convergence
\begin{equation}\label{Lemma:WeakConvergenceU1ToGaussian:TruncatedQuadraticVariation}
\begin{array}{l}
m^{\varepsilon,\eta}(R)
\\
:=
\mathbf{E}  \displaystyle{\left|\int_0^t \nabla_Y^T u_k(X^{\varepsilon,\eta}_s, Y^{\varepsilon,\eta}_s)
\Sigma_1(X^{\varepsilon,\eta}_s)\Sigma_1^T(X^{\varepsilon,\eta}_s)
\nabla_Y u_k(X^{\varepsilon,\eta}_s, Y^{\varepsilon,\eta}_s)\mathbf{1}_{|Y^{\varepsilon,\eta}_s|_{\mathbb{R}^m}\leq R}ds\right.}
\\
\hspace{-.3in}
\displaystyle{\qquad -\left.\int_0^t \overline{\nabla_Y^T u_k(X^{\varepsilon}_s, Y)\mathbf{1}_{|Y|_{\mathbb{R}^m}\leq R}}^\varepsilon(X^\varepsilon_s)
\Sigma_1(X^{\varepsilon}_s)\Sigma_1^T(X^{\varepsilon}_s)
 \overline{\nabla_Y u_k(X^{\varepsilon}_s, Y)\mathbf{1}_{|Y|_{\mathbb{R}^m}\leq R}}^\varepsilon(X^\varepsilon_s)ds\right|}
\\
\displaystyle{\qquad \qquad \rightarrow 0}
\end{array}
\end{equation}
as $\eta \rightarrow 0$.

Therefore by using \eqref{Lemma:PardouxVeretennnikovAPriorBoundsAuxiliaryPDE:Eq:Bounds} for $\nabla_Y u$ we have the estimate
$$\begin{array}{ll}
\hspace*{-5pt}& \mathbf{E}|a_{k,k}^{\varepsilon,\eta}(t)-a_{k,k}^\varepsilon(t)|
\\
\hspace*{-5pt}\leq & m^{\varepsilon,\eta}(R)
\hspace*{-5pt}\\&
\hspace{-.2in}
+
\mathbf{E}  \displaystyle{\left|\int_0^t \nabla_Y^T u_k(X^{\varepsilon,\eta}_s, Y^{\varepsilon,\eta}_s)
\Sigma_1(X^{\varepsilon,\eta}_s)\Sigma_1^T(X^{\varepsilon,\eta}_s)
\nabla_Y u_k(X^{\varepsilon,\eta}_s, Y^{\varepsilon,\eta}_s)\mathbf{1}_{|Y^{\varepsilon,\eta}_s|_{\mathbb{R}^m}> R}ds\right.}
\\
\hspace*{-5pt}& \displaystyle{\hspace{-.2in}
-\left.\int_0^t \overline{\nabla_Y^T u_k(X^{\varepsilon}_s, Y)\mathbf{1}_{|Y|_{\mathbb{R}^m}> R}}^\varepsilon(X^\varepsilon_s)
\Sigma_1(X^{\varepsilon}_s)\Sigma_1^T(X^{\varepsilon}_s)
 \overline{\nabla_Y u_k(X^{\varepsilon}_s, Y)\mathbf{1}_{|Y|_{\mathbb{R}^m}>R}}^\varepsilon(X^\varepsilon_s)ds\right|}
\\
\hspace*{-5pt}\leq &\hspace{-.1in}
 m^{\varepsilon,\eta}(R) + \displaystyle{C\int_0^t \mathbf{E}(1+|Y_s^{\varepsilon,\eta}|_{\mathbb{R}^m}^p)\mathbf{1}_{|Y^{\varepsilon,\eta}_s|_{\mathbb{R}^m}> R}ds}
\\
\hspace*{-5pt}& \displaystyle{\hspace{-.3in}
 \ +\mathbf{E} \int_0^t \overline{\nabla_Y^T u_k(X^{\varepsilon}_s, Y)\mathbf{1}_{|Y|_{\mathbb{R}^m}> R}}^\varepsilon(X^\varepsilon_s)
\Sigma_1(X^{\varepsilon}_s)\Sigma_1^T(X^{\varepsilon}_s)
 \overline{\nabla_Y u_k(X^{\varepsilon}_s, Y)\mathbf{1}_{|Y|_{\mathbb{R}^m}>R}}^\varepsilon(X^\varepsilon_s)ds}
\\
\hspace*{-5pt}\leq & m^{\varepsilon,\eta}(R)+\rho(R) \ .
\end{array}$$
By Cauchy--Schwarz inequality we can estimate
$\left(\mathbf{E}(1+|Y_s^{\varepsilon,\eta}|_{\mathbb{R}^m}^p)\mathbf{1}_{|Y^{\varepsilon,\eta}_s|_{\mathbb{R}^m}> R}\right)^2\leq
\mathbf{E}(1+|Y_s^{\varepsilon,\eta}|_{\mathbb{R}^m}^p)^2\mathbf{P}(|Y_s^{\varepsilon,\eta}|_{\mathbb{R}^m}>R)\rightarrow 0$ as $R\rightarrow \infty$ due to Chebyshev's inequality and estimate \eqref{Lemma:MomentBoundOUCerraiKhasminskii:Eq:FixTimeBoundsY}. Similarly  $\overline{\nabla_Y^T u_k(X^{\varepsilon}_s, Y)\mathbf{1}_{|Y|_{\mathbb{R}^m}> R}}^\varepsilon(X^\varepsilon_s)\rightarrow 0$
with probability 1 as $R\rightarrow\infty$. Thus $\rho(R)\rightarrow 0$ as $R\rightarrow \infty$.
Letting $\eta \rightarrow 0$ first and then $R\rightarrow \infty$ we see that $\mathbf{E}|a_{k,k}^{\varepsilon,\eta}(t)-a_{k,k}^\varepsilon(t)|\rightarrow 0$ as $\eta \rightarrow 0$. Thus
 $[U_1^{\varepsilon,\eta}]_k(t)$ converges weakly in $\mathbf{C}([0,T]; \mathbb{R}^n)$ to a Gaussian process
$[N_1^{\varepsilon}]_k(t)$ with mean $0$ and variance $a_{k,k}^\varepsilon(t)$.

In a same fashion, we can define for any $i,j=1,2,...,n$, that
$$a_{i,j}^{\varepsilon}(t)=\displaystyle{\int_0^t  \overline{\nabla_Y^T u_i(X^{\varepsilon}_s, Y)}^\varepsilon(X^\varepsilon_s)
\Sigma_1(X^{\varepsilon}_s)\Sigma_1^T(X^{\varepsilon}_s)
 \overline{\nabla_Y u_j(X^{\varepsilon}_s, Y)}^\varepsilon(X^\varepsilon_s)ds} \ .$$
Let $A^\varepsilon(t)=(a_{i,j}^{\varepsilon}(t))_{1\leq i,j\leq n}$. With the same reasoning as above, we can show that $U_1^{\varepsilon,\eta}(t)$
converges weakly to $N_1^\varepsilon(t)$ with covariance matrix $A^\varepsilon(t)$.
\end{proof}

In regards to the remark made after Proposition \ref{Proposition:ConvergenceXEpsEtaToXEps}, we can improve the estimate
\eqref{Proposition:ConvergenceXEpsEtaToXEps:Eq:ClosenessXEpsEtaToXEps} in the following lemma.

\begin{lemma}\label{Lemma:ClosenessXToXhatCorrectorMethod}
For any $T>0$ and $\varepsilon>0,\eta>0$ small enough, for $0\leq t\leq T$ and any small $0<\kappa<1$ we have
\begin{equation}\label{Lemma:ClosenessXToXhatCorrectorMethod:Eq:ClosenessXEpsEtaToXEpsCorrector}
\mathbf{E}|X^{\varepsilon,\eta}_t-X^\varepsilon_t|_{\mathbb{R}^n}^2\leq C\left(\dfrac{\eta^2}{\varepsilon^2}+\eta\right) \ ,
\end{equation}
for some constant $C=C(T)>0$.
\end{lemma}

\begin{proof}
We can write the equation \eqref{Eq:DiffusionLimitTimeChangedStandardForm} for $X_t^{\varepsilon,\eta}$ in integral form as
$$X_t^{\varepsilon,\eta}=x_0+\int_0^t B_2(X_s^{\varepsilon,\eta}, Y_s^{\varepsilon,\eta})ds+\sqrt{\eta}\int_0^t \Sigma_2(X_s^{\varepsilon,\eta}, Y_s^{\varepsilon,\eta})dW_s^2 \ .$$

By \eqref{Eq:AveragedSlowMotionGaussianMeasureDeterministic}, we also have
$$X_t^{\varepsilon}=x_0+\int_0^t \overline{B_2(X^\varepsilon_s, Y)}^\varepsilon(X^\varepsilon_s)ds \ .$$

Therefore we can write
$$\begin{array}{ll}
& X_t^{\varepsilon,\eta}-X_t^\varepsilon
\\
=&
\displaystyle{\int_0^t [B_2(X_s^{\varepsilon,\eta}, Y_s^{\varepsilon,\eta})-\overline{B_2(X_s^{\varepsilon,\eta}, Y)}^\varepsilon(X_s^{\varepsilon,\eta})]ds}
\\
& \ \ \ \ \ \ \ \ \ \ +  \displaystyle{\int_0^t [\overline{B_2(X_s^{\varepsilon,\eta}, Y)}^\varepsilon(X_s^{\varepsilon,\eta})-\overline{B_2(X^{\varepsilon}_s, Y)}^\varepsilon(X^{\varepsilon}_s)]ds}
\\
& \ \ \ \ \ \ \ \ \ \ + \displaystyle{\sqrt{\eta}\int_0^t \Sigma_2(X^{\varepsilon,\eta}_s, Y^{\varepsilon,\eta}_s)dW_s^2}
\\
=& (I)+(II)+(III) \ .
\end{array}$$

We can estimate, by \eqref{Lemma:WeakConvergenceU1ToGaussian:Eq:EstimatesOfCorrection} and same methods in the proof of Lemma \ref{Lemma:WeakConvergenceU1ToGaussian}, that for $0\leq t \leq T$ and some constant $C=C(T)$ we have
$$\mathbf{E}|(I)|_{\mathbb{R}^n}^2\leq C \left(\dfrac{\eta^2}{\varepsilon^2}+\eta\right) \ .$$
By Lemma \ref{Lemma:RegularityBarEpsilonOperator} we know that
$$\mathbf{E}|(II)|_{\mathbb{R}^n}^2\leq C \displaystyle{\int_0^t \mathbf{E}|X_s^{\varepsilon,\eta}-X_s^{\varepsilon}|_{\mathbb{R}^n}^2ds \ .}$$
It is straightforward to have
$$\mathbf{E}|(III)|_{\mathbb{R}^n}^2\leq C \eta \ .$$

Combining the above three estimates and make use of Gronwall's inequality, we arrive at \eqref{Lemma:ClosenessXToXhatCorrectorMethod:Eq:ClosenessXEpsEtaToXEpsCorrector}.
\end{proof}

\begin{proposition}\label{Proposition:NormalApproximationWeakConvergence}
As $\eta \rightarrow 0$ the process $Z^{\varepsilon,\eta}_t$ converges weakly on the interval $[0,T]$ and in the space $\mathbf{C}([0,T]; \mathbb{R}^n)$
to the process $Z^{\varepsilon}_t$ defined by the following equation
\begin{equation}\label{Proposition:NormalApproximationWeakConvergence:Eq:LimitingProcess}
Z^\varepsilon_t=\int_0^t M(X^\varepsilon_s)Z^\varepsilon_sds+N_1^\varepsilon(t)+N_2^\varepsilon(t) \ ,
\end{equation}
where $N_1^\varepsilon(t)$ and $N_2^\varepsilon(t)$ are two Gaussian processes with means $0$ and explicitly calculated covariances, and
$M(X)$ is an $n\times n$ matrix function.
\end{proposition}

\begin{proof}
Apparently, from \eqref{Eq:RescaledDeviationIntegralEquation}, \eqref{Eq:UProcess}, \eqref{Eq:VProcess}, \eqref{Eq:U1Process} and
\eqref{Eq:U2Process} we know that we have the decomposition
$$Z^{\varepsilon,\eta}(t)=U_2^{\varepsilon,\eta}(t)+U_1^{\varepsilon,\eta}(t)+V^{\varepsilon,\eta}(t) \ .$$

By using Lemma \ref{Lemma:AveragingPrincipleIntegralLimit}, we know that the process $V^{\varepsilon,\eta}(t)$ converges weakly as $\eta \rightarrow 0$
to a Gaussian process $N_2^\varepsilon(t)$ with covariance matrix
$$\mathcal{A}^\varepsilon(t)=\left(\displaystyle{\int_0^t \overline{\Sigma_2(X^\varepsilon_s, Y)}^\varepsilon(X^\varepsilon_s)
 \overline{\Sigma_2^T(X^\varepsilon_s, Y)}^\varepsilon(X^\varepsilon_s)ds}\right)_{1\leq i,j\leq n} \ .$$

Lemma \ref{Lemma:WeakConvergenceU1ToGaussian} provides the weak convergence of $U_1^{\varepsilon,\eta}(t)$ to $N_1^\varepsilon(t)$
with covariance matrix $A^\varepsilon(t)$ as in \eqref{Lemma:WeakConvergenceU1ToGaussian:Eq:CovarianceMatrix}.

Finally, the weak convergence of $U_2^{\varepsilon,\eta}(t)$ to $\displaystyle{\int_0^t M(X^\varepsilon_s)Z_s^\varepsilon ds}$ can be obtained by doing a standard
Taylor expansion argument as in the proof of Theorem 3.1 of \cite{[Khasminskii1966SmallParameterDE]}. In fact, set
$M(X)=\nabla_X[\overline{B_2(X,Y)}^\varepsilon(X)]$. From \eqref{Eq:U2Process} we have
$$\begin{array}{ll}
&U_2^{\varepsilon,\eta}(t)-\displaystyle{\int_0^t M(X^\varepsilon_s)Z^{\varepsilon,\eta}_sds}
\\
=&\displaystyle{\dfrac{1}{\sqrt{\eta}}\int_0^t \left[\overline{B_2(X^{\varepsilon,\eta}_s,Y)}^{\varepsilon}(X^{\varepsilon,\eta}_s)
-\overline{B_2(X^{\varepsilon}_s,Y)}^{\varepsilon}(X^{\varepsilon}_s)-\sqrt{\eta}M(X^\varepsilon_s)Z^{\varepsilon,\eta}_s\right]ds }
\\
=&\displaystyle{\dfrac{1}{\sqrt{\eta}}\int_0^t \left[
\overline{B_2(X^{\varepsilon}_s+\sqrt{\eta} Z^{\varepsilon,\eta}_s,Y)}^{\varepsilon}(X^{\varepsilon}_s+\sqrt{\eta}Z^{\varepsilon,\eta}_s)
\right.}
\\&\hspace{1in}
\displaystyle{\left.
-\overline{B_2(X^{\varepsilon}_s,Y)}^{\varepsilon}(X^{\varepsilon}_s)
-\sqrt{\eta}M(X^\varepsilon_s)Z^{\varepsilon,\eta}_s\right]ds \ .}
\end{array}$$

Therefore by boundedness of second derivatives of of $B_2(X,Y)$ with respect to $X$ and reasoning as in Lemma \ref{Lemma:RegularityBarEpsilonOperator},
we have
\begin{equation}\label{Proposition:NormalApproximationWeakConvergence:Eq:ClosenessU2DeterministicPart}
\mathbf{E}\left|U_2^{\varepsilon,\eta}(t)-\displaystyle{\int_0^t M(X^\varepsilon_s)Z^{\varepsilon,\eta}_sds}\right|_{\mathbb{R}^n}
\leq C\sqrt{\eta}\int_0^t \mathbf{E}|Z^{\varepsilon,\eta}_s|^2_{\mathbb{R}^n} ds \ .
\end{equation}

By making use of Lemma \ref{Lemma:ClosenessXToXhatCorrectorMethod},
as well as Lemma \ref{Lemma:RegularityBarEpsilonOperator}, we have
$$\begin{array}{ll}
&\mathbf{E}|U_2^{\varepsilon,\eta}(t)|_{\mathbb{R}^n}^2
\\
=& \displaystyle{\dfrac{1}{\eta}\mathbf{E}\left|\int_0^t [\overline{B_2(X_s^{\varepsilon,\eta}, Y)}^\varepsilon(X_s^{\varepsilon,\eta})-\overline{B_2(X^\varepsilon_s, Y)}^\varepsilon(X_s^\varepsilon)]ds\right|_{\mathbb{R}^n}^2}
\\
\leq & \displaystyle{\dfrac{C}{\eta}\int_0^t \mathbf{E}|X_s^{\varepsilon,\eta}-X_s^\varepsilon|_{\mathbb{R}^n}^2ds}
\\
\leq & C<\infty
\end{array}$$
and
$$\begin{array}{ll}
& \mathbf{E}|U_2^{\varepsilon,\eta}(t+h)-U_2^{\varepsilon,\eta}(t)|_{\mathbb{R}^n}^2
\\
= & \displaystyle{\dfrac{1}{\eta}\mathbf{E}\left|\int_{t}^{t+h} [\overline{B_2(X_s^{\varepsilon,\eta}, Y)}^\varepsilon(X_s^{\varepsilon,\eta})-\overline{B_2(X^\varepsilon_s, Y)}^\varepsilon(X_s^\varepsilon)]ds\right|_{\mathbb{R}^n}^2}
\\
\leq & \displaystyle{\dfrac{Ch}{\eta}\int_{t}^{t+h} \mathbf{E}|X_s^{\varepsilon,\eta}-X_s^\varepsilon|_{\mathbb{R}^n}^2ds}
\\
\leq & Ch^2 \ ,
\end{array}$$
which then imply $\mathbf{E}|Z^{\varepsilon,\eta}_t|_{\mathbb{R}^n}^2\leq C<\infty$, as well as the tighness of the family $Z^{\varepsilon,\eta}(t)$
in $\mathbf{C}([0,T]; \mathbb{R}^n)$. From the weak convergence of $U_1^{\varepsilon,\eta}(t)$ and $V^{\varepsilon,\eta}(t)$
to Gaussian processes, together with \eqref{Proposition:NormalApproximationWeakConvergence:Eq:ClosenessU2DeterministicPart}, we conclude
this Proposition.
\end{proof}

\section{Error estimate of SCGD from averaged SGD: Justification of the approximation.}\label{sec:error}

We briefly mention the justification of using SGD \eqref{Eq:AveragedSlowMotionSGD} to approximate $X^{\varepsilon,\eta}(t)$
in \eqref{Eq:DiffusionLimitTimeChangedStandardForm}. From Lemma
\ref{Lemma:ErrorSmallEpsAveragedQuantity} we know that as $\varepsilon\rightarrow 0$, by \eqref{Eq:Quantity:B2},
$$\overline{B_2(X, Y)}^\varepsilon(X)\approx -\mathbf{E} \widetilde{\nabla} g_w(X)\nabla f_v(\mathbf{E} g_w(X))+\mathcal{O}(\sqrt{\varepsilon}) \ .$$
Thus as $\varepsilon \rightarrow 0$, the process $X^\varepsilon(t)$ approximates another process $\bar{X}(t)$ that solves an ordinary differential equation
\eqref{Eq:AveragedSlowMotionGD}:
$$
d\bar{X}(t)=-\mathbf{E} \widetilde{\nabla} g_w(\bar{X}(t))\nabla f_v(\mathbf{E} g_w(\bar{X}(t)))dt \ , \ \bar{X}(0)=x_0 \ ,
$$
with an error of $\mathcal{O}(\sqrt{\varepsilon})$. In fact, equation \eqref{Eq:AveragedSlowMotionGD} can be viewed as a continuous version of the
Gradient Descent (GD) algorithm, which directly solves \eqref{Eq:OptimizationProblem}.

Furthermore, from Proposition \ref{Proposition:NormalApproximationWeakConvergence} we know that,
 as $\eta \rightarrow 0$, the process $Z^{\varepsilon,\eta}(t)$ converges weakly to random process $Z_t^\varepsilon$.
 The process $Z_t^\varepsilon$ has its deterministic drift part and is driven by two mean $0$ Gaussian processes
carrying explicitly calculated covariance structures.
This implies that, roughly speaking, from \eqref{Eq:RescaledDeviation} we have an expansion of the type
\eqref{Eq:ExpansionNormalApproximation}:
$$
X^{\varepsilon,\eta}(t)\stackrel{\mathcal{D}}{\approx} X^\varepsilon(t)+\sqrt{\eta}Z_t^\varepsilon \ ,
$$
as $\eta\rightarrow 0$. Here $\stackrel{\mathcal{D}}{\approx}$ means approximate equality of probability distributions.

Therefore by \eqref{Eq:ConvergenceSmallEpsAveragedQuantity}, \eqref{Eq:AveragedSlowMotionGD} and \eqref{Eq:ExpansionNormalApproximation}
we know that the slow motion $X^{\varepsilon,\eta}(t)$ in \eqref{Eq:DiffusionLimitTimeChangedStandardForm} (or \eqref{Eq:DiffusionLimitTimeChanged})
has an expansion around the GD algorithm in \eqref{Eq:AveragedSlowMotionGD}:
$$
X^{\varepsilon,\eta}(t)\stackrel{\mathcal{D}}{\approx} \bar{X}(t)+\mathcal{O}(\sqrt{\varepsilon})+\sqrt{\eta}Z_t^\varepsilon \ .
$$

Let us introduce the process $\mathrm{X}^{\varepsilon,\eta}(t)$ as
the following continuous version of the Stochastic Gradient Descent (SGD) algorithm, as in \eqref{Eq:AveragedSlowMotionSGD}:
$$
d\mathrm{X}^{\varepsilon,\eta}(t)=
-\mathbf{E} \widetilde{\nabla} g_w(\mathrm{X}^{\varepsilon,\eta}(t))\nabla f_v(\mathbf{E} g_w(\mathrm{X}^{\varepsilon,\eta}(t)))dt+\sqrt{\eta}dZ_t^\varepsilon \ , \
 \mathrm{X}^{\varepsilon,\eta}(0)=x_0 \ .
$$

From \eqref{Eq:AveragedSlowMotionGD} and \eqref{Eq:AveragedSlowMotionSGD} and using Gronwall's inequality, we know that
$$\bar{X}(t)-\mathrm{X}^{\varepsilon,\eta}(t)\approx \mathcal{O}(\sqrt{\eta}) \ .$$
So that by \eqref{Eq:ExpansionNormalApproximationGD} we further have
$$
X^{\varepsilon,\eta}(t)\stackrel{\mathcal{D}}{\approx} \mathrm{X}^{\varepsilon,\eta}(t)+\mathcal{O}(\sqrt{\varepsilon})+\mathcal{O}(\sqrt{\eta}) \ .
$$

The above is a justifiation of the approximation of SGD to the process $X^{\varepsilon,\eta}_t$ in \eqref{Eq:DiffusionLimitTimeChangedStandardForm}.
In the strongly convex case, $\mathrm{X}^{\varepsilon,\eta}(t)$ in \eqref{Eq:AveragedSlowMotionSGD} enters a small neighborhood containing the minimizer of
 \eqref{Eq:OptimizationProblem} in finite time $T>0$, so that
\eqref{Eq:ExpansionNormalApproximationSGD} implies $X^{\varepsilon,\eta}(t)$ in \eqref{Eq:DiffusionLimitTimeChanged} enters a basin containing the minimizer of  \eqref{Eq:OptimizationProblem}
also in finite time $T>0$. This validates the effectiveness of using the SCGD algorithm in the strongly convex case.

\section{Remarks and generalizations.}\label{sec:remarks}

(a) For general fast--slow systems of stochastic differential equations, strong approximation theorems are available
(see \cite{[Bakhtin2003]}, \cite{[BakhtinKifer2004]}).
Let us introduce a diffusion approximation $\mathrm{\underline{X}}^{\varepsilon,\eta}(t)$ of $X^\varepsilon(t)$ in \eqref{Eq:AveragedSlowMotionGaussianMeasureDeterministic}
by the stochastic differential equation:
\begin{equation}\label{Remark:Eq:AveragedSlowMotionGaussianMeasure}
d\mathrm{\underline{X}}^{\varepsilon,\eta}(t)=\overline{B_2(\mathrm{\underline{X}}^{\varepsilon,\eta}(t), Y)}^\varepsilon(\mathrm{\underline{X}}^{\varepsilon,\eta}(t))dt
+\sqrt{\eta}\Sigma^\varepsilon(\mathrm{\underline{X}}^{\varepsilon,\eta}(t))dW_t^2 \ , \ \mathrm{\underline{X}}^{\varepsilon,\eta}(0)=x_0 \ .
\end{equation}

Here $\Sigma^\varepsilon(X)\in \mathbb{R}^n\otimes \mathbb{R}^n$ is some appropriately chosen
non--degenerate noise matrix. The method of Bakhtin--Kifer (see \cite{[BakhtinKifer2004]}) provides a more refined
diffusion approximation analysis than the classical averaging principle. Roughly speaking, we have
for $0\leq t \leq T$,
\begin{equation}\label{Remark:Eq:Bakhtin-KiferDiffusionApproximation}
\mathbf{E}|X^{\varepsilon,\eta}(t)-\mathrm{\underline{X}}^{\varepsilon,\eta}(t)|_{\mathbb{R}^n}^2\leq C\eta^{1+\delta} \ ,
\end{equation}
for some $C=C(T)>0$ and small $\delta>0$.

Since we have Lemma \ref{Lemma:ErrorSmallEpsAveragedQuantity},
we shall also consider the diffusion limit $\mathrm{\underline{X}}^\eta(t)$ under the following SGD algorithm:
\begin{equation}\label{Remark:Eq:AveragedSlowMotionSGD}
d\mathrm{\underline{X}}^\eta(t)=-\mathbf{E} \widetilde{\nabla} g_w(\mathrm{\underline{X}}^\eta(t))\nabla f_v(\mathbf{E} g_w(\mathrm{\underline{X}}^\eta(t)))dt
+\sqrt{\eta}\Sigma(\mathrm{\underline{X}}^\eta(t))dW_t^2 \ , \ \mathrm{\underline{X}}^\eta(0)=x_0 \ ,
\end{equation}
where $\Sigma(X)\in \mathbb{R}^n\otimes \mathbb{R}^n$ is some appropriately chosen non--degenerate noise matrix.

The above diffusion limit of SGD algorithm aims at directly solving the optimization problem \eqref{Eq:OptimizationProblem}.
The convergence time analysis
in terms of $\eta$ follows standard results in SGD convergence analysis.
By using standard technique in the theory of stochastic differential equations,
we have, roughly speaking, for $0\leq t \leq  T$,
\begin{equation}\label{Remark:Eq:SecondStepDiffusionApproximation}
\mathbf{E}|\mathrm{\underline{X}}^{\varepsilon,\eta}(t)-\mathrm{X}^\eta(t)|_{\mathbb{R}^n}^2\leq C(\varepsilon+\eta) \ ,
\end{equation}
for some $C=C(T)>0$ and small $\delta>0$.

Combining \eqref{Remark:Eq:Bakhtin-KiferDiffusionApproximation} and \eqref{Remark:Eq:SecondStepDiffusionApproximation}, we obtain
an error bound, that for $0\leq t \leq T$,
\begin{equation}\label{Remark:Eq:ErrorBound}
\mathbf{E}|X^{\varepsilon,\eta}(t)-\mathrm{X}^\eta(t)|_{\mathbb{R}^n}^2\leq C (\varepsilon+\eta+\eta^{1+\delta}) \ ,
\end{equation}
for some $C=C(T)>0$ and small $\delta>0$.

However, the method provided by Bakhtin--Kifer can only cover the case when fast motion is moving on a compact space. Yet in our case the fast motion
$Y^{\varepsilon,\eta}(t)$ is an OU process for frozen $X$. Thus estimates \eqref{Remark:Eq:Bakhtin-KiferDiffusionApproximation} and henceforth \eqref{Remark:Eq:ErrorBound}
are only conjectures and have to be addressed in a future work.

(b) By using the corrector method as we did in Section 3, it is possible to show that the order of approximation
in Proposition \ref{Proposition:ConvergenceXEpsEtaToXEps} can be improved to be $C\left(\dfrac{\eta^2}{\varepsilon^2}+\eta\right)$,
as in Lemma \ref{Lemma:ClosenessXToXhatCorrectorMethod}.
However, our normal deviation analysis indicates that the order of
approximation of $X^{\varepsilon}(t)$ to $X^{\varepsilon,\eta}(t)$ in mean square sense has to be of order $\mathcal{O}(\eta)$ (see \eqref{Eq:ExpansionNormalApproximation}).
This is because the latter approximation is only in the weak sense, and the former approximation is in the strong ($L^2$) sense.
Thus the weak approximation can achieve better convergence rates. On the other hand,
by making use of Dambis--Dubins--Schwarz theorem (see \cite[Theorem 1.6]{[RY99]}), as well as the H\"{o}lder
continuity of the Brownian motion path, it is also possible to show strong approximations in the normal deviation analysis (see \cite{[HairerPavliotis2004]}).

(c) We have assumed that the functions $f_v: \mathbb{R}^m\rightarrow \mathbb{R}$ and $g_w: \mathbb{R}^n\rightarrow \mathbb{R}^m$ are supported on some compact subsets of $\mathbb{R}^m$
and $\mathbb{R}^n$, respectively. This leads to the fact that the drift vector fields and diffusion matrix fields $B_1(X)$, $B_2(X,Y)$,
$A_1(X)$, $A_2(X,Y)$ in \eqref{Eq:Quantity:B1}, \eqref{Eq:Quantity:B2}, \eqref{Eq:Quantity:A1}, \eqref{Eq:Quantity:A2}
contain bounded coefficients together with their first derivatives. Such an assumption is essential for our arguments in
 deriving the averaging principle and normal deviation results.
In practical situations, as we are dealing with the optimization problem \eqref{Eq:OptimizationProblem}, we are only interested in the dynamics
of the corresponding algorithm trajectories that approach the minimizer. Thus we can take a large ball in
the Euclidean space containing this minimizer, and we eliminate the trajectories outside this ball. By making use of large deviation estimates,
this leads to the fact that all error bounds or approximation results in our work could only be understood to be valid with
high probability (with probability close to $1$). The slogan of deriving approximation or convergence results with high
probability is consistent with standard results in the statistical machine learning literature (see \cite{[Vapnik]}).

\appendix

\section{Two technical lemmas.}

The following lemma characterizes quantitatively the convergence
$\overline{q(X,Y)}^\varepsilon\rightarrow q(X, \mathbf{E} g_w(X)) \text{ as } \varepsilon\rightarrow 0$.
Recall that the object $q$ in \eqref{Eq:BarEpsOperator} may be a scalar, a vector, or a matrix.
Let, in general, the components of $q$ be $q_{ij}$. Let
$$\sup\limits_{X,Y}\left(
\sum\limits_{k=1}^{n}\sum\limits_{i,j}\left|\dfrac{\partial q_{ij}}{\partial X_k}(X,Y)\right| \ , \
\sum\limits_{l=1}^{m}\sum\limits_{i,j}\left|\dfrac{\partial q_{ij}}{\partial Y_l}(X,Y)\right|\right)\leq M \ .$$

\begin{lemma}\label{Lemma:ErrorSmallEpsAveragedQuantity}
We have
$$|\overline{q(X,Y)}^\varepsilon- q(X, \mathbf{E} g_w(X))|_{\text{norm}}\leq C\sqrt{\varepsilon} \ ,$$
where the constant $C>0$ depends on $M$, and the norm $|q|_{\text{norm}}$ is a vector (matrix) norm if $q$ is a vector (matrix), respectively.
\end{lemma}

\begin{proof}
The Gaussian measure $\mu^{X,\varepsilon}(dY)$ in \eqref{Eq:InvariantMeasureOUSmallDiffusion} has a density function
(see \cite[Theorem 1.2.9]{[Muirhead]})
$$\mu^{X,\varepsilon}(dY)=\dfrac{\exp\left(-\dfrac{1}{2}(Y-\mathbf{E} g_w(X))^T\left(\dfrac{\varepsilon}{2}\Sigma_1(X)\Sigma_1(X)^T\right)^{-1}(Y-\mathbf{E} g_w(X))\right)}
{(2\pi)^{\frac{m}{2}}\left(\dfrac{\varepsilon}{2}\right)^{m/2}[\det(\Sigma_1(X)\Sigma_1(X)^{T})]^{1/2}}
dY \ .$$

Let the density function
$$\rho(Z)=\dfrac{\exp\left(-\dfrac{1}{2}Z^T (\Sigma_1(X)\Sigma_1(X)^T)^{-1}Z\right)}{(2\pi)^{\frac{m}{2}}[\det(\Sigma_1(X)\Sigma_1(X)^T)]^{1/2}} \ ,$$
so that $$\int_{\mathbb{R}^m}\rho(Z)dZ=1 \ .$$

Then we have
$$\begin{array}{ll}
& \displaystyle{\int_{\mathbb{R}^m} q(X,Y)\mu^{X,\varepsilon}(dY)}
\\
= & \hspace{-.1in}
\displaystyle{\int_{\mathbb{R}^m}\dfrac{\exp\left(-\dfrac{1}{2}(Y-\mathbf{E} g_w(X))^T\left(\dfrac{\varepsilon}{2}\Sigma_1(X)\Sigma_1(X)^T\right)^{-1}(Y-\mathbf{E} g_w(X))\right)}
{(2\pi)^{\frac{m}{2}}\left(\dfrac{\varepsilon}{2}\right)^{m/2}[\det(\Sigma_1(X)\Sigma_1(X)^{T})]^{1/2}} q(X,Y)
dY}
\\= & \hspace{-.1in}
\displaystyle{\int_{\mathbb{R}^m} \rho(Z) q\left(X, \mathbf{E} g_w(X)+\sqrt{\dfrac{\varepsilon}{2}}Z\right)dZ \ .}
\end{array}$$

From here we have
\begin{align*}
&\quad
\displaystyle{\int_{\mathbb{R}^m} q(X,Y)\mu^{X,\varepsilon}(dY)-q(X, \mathbf{E} g_w(X))}
\\&=
\displaystyle{\int_{\mathbb{R}^m} \rho(Z) \left[q\left(X, \mathbf{E} g_w(X)+\sqrt{\dfrac{\varepsilon}{2}}Z\right)-q(X, \mathbf{E} g_w(X))\right]dZ \ ,}
\end{align*}
so that
$$\displaystyle{\left|\int_{\mathbb{R}^m} q(X,Y)\mu^{X,\varepsilon}(dY)-q(X, \mathbf{E} g_w(X))\right|_{\text{norm}}}
\leq \sqrt{\dfrac{\varepsilon}{2}}M \int_{\mathbb{R}^m}|Z|_{\mathbb{R}^m}\rho(Z)dZ\leq C\sqrt{\varepsilon} \ .$$
\end{proof}

The following Lemma is about regularity properties with respect to $X$ of the $\overline{q(X,Y)}^\varepsilon(X)$ operator.

\begin{lemma}\label{Lemma:RegularityBarEpsilonOperator}
For any $X_1, X_2\in \mathbb{R}^n$ and some $C>0$ we have
\begin{equation}\label{Lemma:RegularityBarEpsilonOperator:Eq:LipschitzContinuityBarEpsOperator}
|\overline{q(X_1,Y)}^\varepsilon(X_1)-\overline{q(X_2,Y)}^\varepsilon(X_2)|_{\text{norm}}\leq C|X_1-X_2|_{\mathbb{R}^n}  \ ,
\end{equation}
where the constant $C>0$ depends on $M$, and the norm $|q|_{\text{norm}}$ is a vector (matrix) norm if $q$ is a vector (matrix), respectively.
\end{lemma}

\begin{proof}
The Gaussian measure $\mu^{X,\varepsilon}(dY)$ in \eqref{Eq:InvariantMeasureOUSmallDiffusion} has a density function
(see \cite[Theorem 1.2.9]{[Muirhead]})
$$\mu^{X,\varepsilon}(dY)=\dfrac{\exp\left(-\dfrac{1}{2}(Y-\mathbf{E} g_w(X))^T\left(\dfrac{\varepsilon}{2}\Sigma_1(X)\Sigma_1(X)^T\right)^{-1}(Y-\mathbf{E} g_w(X))\right)}
{(2\pi)^{\frac{m}{2}}\left(\dfrac{\varepsilon}{2}\right)^{m/2}[\det(\Sigma_1(X)\Sigma_1(X)^{T})]^{1/2}}
dY \ .$$

Let the density function for the standard normal distribution $\mathcal{N}(0, I_{m})$ be
$$\mu(dN)=\dfrac{1}{(2\pi)^{\frac{m}{2}}}\exp\left(-\dfrac{1}{2}N^T N\right)dN \ ,$$
so that
$$\int_{\mathbb{R}^m}\mu(dN)=1 \ .$$

Let the random variable $N$ follow the standard normal distribution $\mathcal{N}(0, I_m)$, so that $N$ has the distribution $\mu(dN)$
on $\mathbb{R}^m$. If $Y$ is a random variable that follows the distribution $\mu^{X,\varepsilon}(dY)$, then
$Y=\sigma(X)N+\mathbf{E} g_w(X)$, where $\sigma(X)$ is a non--singular $m\times m$ matrix such that $\sigma(X)\sigma^T(X)=(\Sigma_1(X)\Sigma_1^T(X))^{-1}$
(see \cite[Theorem 1.2.9]{[Muirhead]}). By \cite[\S 3.2, Theorem 2.1]{[FreidlinRedBook]},
the matrix $\sigma(X)$ can be chosen to be symmetric and Lipschitz continuous. Since $\Sigma_1(X)\Sigma_1^T(X)$ is compactly supported,
the matrix $\sigma(X)$ can also be chosen to be compactly supported.

Then we have
{\small$$\begin{array}{ll}
& \overline{q(X_1,Y)}^\varepsilon(X_1)-\overline{q(X_2,Y)}^\varepsilon(X_2)
\\
= & \displaystyle{\int_{\mathbb{R}^m} q(X_1,Y)\mu^{X_1,\varepsilon}(dY)-\int_{\mathbb{R}^m} q(X_2, Y)\mu^{X_2,\varepsilon}(dY)}
\\
= & \displaystyle{\int_{\mathbb{R}^m} q(X_1,\sigma(X_1)N+\mathbf{E} g_w(X_1))\mu(dN)-\int_{\mathbb{R}^m} q(X_2,\sigma(X_2)N+\mathbf{E} g_w(X_2))\mu(dN) \ .}
\end{array}$$}

From here we have
\begin{align*}
&\quad
\displaystyle{|\overline{q(X_1,Y)}^\varepsilon(X_1)-\overline{q(X_2,Y)}^\varepsilon(X_2)|_{\text{norm}}}
\\&\leq
C|X_1-X_2|_{\mathbb{R}^n}\displaystyle{\int_{\mathbb{R}^m}|N|_{\mathbb{R}^m}\mu(dN)}\leq C|X_1-X_2|_{\mathbb{R}^n}
\ .
\end{align*}
\end{proof}

\section*{Acknowledgments}
We would like to thank the anonymous referee for valuable suggestions. We also would like to thank
Professor Haoyi Xiong for many relevant discussions.


\medskip
Received for publication October 2017.
\medskip


\begin{thebibliography}{99}

\bibitem{[LArnold]} (MR1948294)
\newblock L. Arnold,
\newblock Hasslemann's program revisited: The analysis of stochasticity in deterministic climate models,
\newblock in \emph{Progress in Probability Book Series}, Springer, \textbf{49}, Stochastic Climate Models, 141--157.

\bibitem{[Bakhtin2003]} (MR2017781) [10.1080/1045112031000155678]
\newblock V. I. Bakhtin,
\newblock \doititle{Asymptotics of superregular perturbations of fiber ergodic semigroups},
\newblock \emph{Stochastics and Stochastic Reports}, \textbf{75} (2003), 295--318.

\bibitem{[BakhtinKifer2004]} (MR2063374) [10.1007/s00440-003-0326-7]
\newblock V. Bakhtin and Y. Kifer,
\newblock \doititle{Diffusion approximation for slow motion in fully coupled averaging},
\newblock \emph{Probability Theory and Related Fields}, \textbf{129} (2004), 157--181.

\bibitem{[AdaptedAlgorithmbook]} (MR1082341) [10.1007/978-3-642-75894-2]
\newblock A. Benveniste, M. Metivier and P. Priouret,
\newblock \doititle{Adaptive algorithms and stochastic approximations},
\newblock \emph{Applications of Mathematics}, Springer, \textbf{22} (1990), xii+365pp.

\bibitem{[Borkar]} (MR2442439)
\newblock V. S. Borkar,
\newblock \emph{Stochastic Approximation: A Dynamical Systems Viewpoint},
\newblock Cambridge University Press, 2008.

\bibitem{[CerraiKhasminskii2009]} (MR2537194) [10.1214/08-AAP560]
\newblock S. Cerrai,
\newblock \doititle{A Khasminskii's averaging principle for stochastic reaction--diffusion equations},
\newblock \emph{Annals of Applied Probability}, \textbf{19} (2009), 899--948.

\bibitem{[CerraiNormalDeviation2009]} (MR2531558) [10.1016/j.matpur.2009.04.007]
\newblock S. Cerrai,
\newblock \doititle{Normal deviations from the averaged motion for some reaction--diffusion equations with fast oscillating perturbation},
\newblock \emph{Journal de Math\'{e}matiques Pures et Appliqu\'{e}es}, \textbf{91} (2009), 614--647.

\bibitem{[Ermoliev]} (MR0444016)
\newblock Y. Ermoliev,
\newblock \emph{Methods of Stochastic Programming. Monographs in Optimization and OR},
\newblock Nauka, Moscow, 1976.

\bibitem{[FreidlinRedBook]} (MR833742) [10.1515/9781400881598]
\newblock M. I. Freidlin,
\newblock \emph{Functional Integration and Partial Differential Equations},
\newblock Princeton University Press, Princeton, 1985.

\bibitem{[FW book]} (MR1652127) [10.1007/978-1-4612-0611-8]
\newblock M. Freidlin and A. Wentzell,
\newblock \emph{Random Perturbations of Dynamical Systems},
\newblock 2$^{nd}$ Edition, Springer, 1998.

\bibitem{[HairerPavliotis2004]} (MR2098568) [10.1023/B:JOSS.0000044055.59822.20]
\newblock M. Hairer and G. Pavliotis,
\newblock \doititle{Periodic homogenization for hypoelliptic diffusions},
\newblock \emph{Journal of Statistical Physics}, \textbf{117} (2004), 261--279.

\bibitem{[Hasselmann1976]}
\newblock K. Hasselmann,
\newblock Stochastic climate models, Part I, Theory,
\newblock \emph{Tellus}, \textbf{28} (1976), 473--485.

\bibitem{[hu2017fast]}
\newblock W. Hu and C. J. Li,
\newblock On the fast convergence of random perturbations of the gradient flow,
\newblock preprint, \arXiv{1706.00837}

\bibitem{[JunchiEtAlDiffusionApproximation]}
\newblock W. Hu, C.J. Li, L. Li and J. Liu,
\newblock On the diffusion approximation of nonconvex stochastic gradient descent,
\newblock \emph{Annals of Mathematical Science and Applications}, to appear.
\arXiv{1705.07562}

\bibitem{[Khasminskii1966SmallParameterDE]}
\newblock R. Khasminskii,
\newblock \doititle{On stochastic processes defined by differential equations with a small parameter},
\newblock \emph{Theory of Probability and its Applications}, \textbf{11} (1966), 211--228.

\bibitem{[Khasminskii1968AveragingSDE]} (MR0260052)
\newblock R. Khasminskii,
\newblock \doititle{On the principle of averaging the It\^{o}'s stochastic differential equations},
\newblock \emph{Kybernetika (Prague)}, \textbf{4} (1968), 260--279.

\bibitem{[Kifer1981SaddlePoint]} (MR636908) [10.1007/BF02761819] [10.1007/BF02761819]
\newblock Y. Kifer,
\newblock \doititle{The exit problem for small random perturbations of dynamical systems with a hyperbolic fixed point},
\newblock \emph{Israel Journal of Mathematics}, \textbf{40} (1981), 74--96.

\bibitem{[KV86]} (MR834478) [10.1007/BF01210789]
\newblock C. Kipnis and S. R. S. Varadhan,
\newblock \doititle{Central limit theorem for additive functionals of reversible Markov processes and applications to simple exclusions},
\newblock \emph{Communications in Mathematical Physics}, \textbf{104} (1986), 1--19.

\bibitem{[Kushner-Yin]}
\newblock H. Kushner and G. George Yin,
\newblock Stochastic approximations and recursive algorithms and applications,
\newblock \emph{Applications of Mathematics (Stochastic Modeling and Applied Probability)}, \textbf{35}, 2$^{nd}$ Edition, Springer, 2003.

\bibitem{[Muirhead]} (MR652932) [10.1109/IGARSS.2010.5653437]
\newblock R. J. Muirhead,
\newblock \emph{Aspects of Multivariate Statistical Theory},
\newblock John Wiley \& Sons, Inc., New York, 1982.

\bibitem{[Oksendal]} (MR1217084) [10.1007/978-3-662-02847-6]
\newblock B.Oksendal,
\newblock \emph{Stochastic differential equations},
\newblock Springer--Verlag, 1992.

\bibitem{[Pardoux-Veretennikov1]} (MR1872736) [10.1214/aop/1015345596]
\newblock E. Pardoux and A. Yu. Veretennikov,
\newblock \doititle{On the Possion equation and diffusion approximation 1},
\newblock \emph{Annals of Probability}, \textbf{29} (2001), 1061--1085.

\bibitem{[Pardoux-Veretennikov2]} (MR1988467) [10.1214/aop/1055425774]
\newblock E. Pardoux and A. Yu. Veretennikov,
\newblock \doititle{On the Possion equation and diffusion approximation 2},
\newblock \emph{Annals of Probability}, \textbf{31} (2003), 1166--1192.

\bibitem{[Pardoux-Veretennikov3]} (MR2135314) [10.1214/009117905000000062]
\newblock E. Pardoux and A. Yu. Veretennikov,
\newblock \doititle{On the Possion equation and diffusion approximation 3},
\newblock \emph{Annals of Probability}, \textbf{33} (2005), 1111--1133.

\bibitem{[RY99]} (MR1725357) [10.1007/978-3-662-06400-9]
\newblock D. Revuz and M. Yor,
\newblock \doititle{Continuous martingales and Brownian motion},
\newblock \emph{Grundlehren der Mathematischen Wissenschaften}, \textbf{293}, 3$^{rd}$ edition, Springer--Verlag, Berlin, 1999, xiv+602 pp.

\bibitem{[VanKampen]}
\newblock N. G. Van Kampen,
\newblock The diffusion approximation for markov processes,
\newblock Reprinted from: \emph{Thermodynamics $\&$ kinetics of biological processes} (eds. I. Lamprecht and A. I. Zotin) Walter de Gruyter $\&$ Co., New York (1982). 181--195.

\bibitem{[Vapnik]} (MR1367965) [10.1007/978-1-4757-2440-0]
\newblock V. N. Vapnik,
\newblock \emph{The Nature of Statistical Learning Theory},
\newblock Springer, 1995.

\bibitem{[Veretennikov1997]} (MR1751475) [10.1137/S0040585X97977550]
\newblock A. Yu. Veretennikov,
\newblock \doititle{On polynomial mixing and convergence rate for stochastic difference and differential equations},
\newblock \emph{Theory of Probability and its Applications}, \textbf{44} (2000), 361--374.

\bibitem{[SCGDpaper]} (MR3592784) [10.1007/s10107-016-1017-3]
\newblock M. Wang, E. X. Fang and H. Liu,
\newblock \doititle{Stochastic compositional gradient descent: Algorithms for minimizing compositions of expected--value functions},
\newblock \emph{Mathematical Programming}, \textbf{161} (2016), 419--449.

\bibitem{[SCGDpaper2]}
\newblock M. Wang, J. Liu and E. X. Fang,
\newblock Accelerating stochastic composition optimization,
\newblock Advances in Neural Information Processing Systems, 2016. \arXiv{1607.07329}


\end{thebibliography}
\end{document}